\documentclass[12pt,a4paper,reqno]{amsart}

\usepackage{amsfonts}
\usepackage{graphicx}
\usepackage{graphics}
\usepackage{epsfig}
\usepackage{anysize}
\usepackage{lipsum}
\usepackage{wrapfig}
\marginsize{3cm}{3cm}{3cm}{3cm}
\newtheorem{theorem}{Theorem}[section]

\newtheorem{definition}{Definition}[section]

\newtheorem{notation}{Notation}[section]

\numberwithin{equation}{section}
\makeatletter

\begin{document}

	\begin{flushleft}
		{\scriptsize }
	\end{flushleft}
	\setcounter{page}{1}
	\vspace{2cm}
	%-----------------------------------------------------------------------------
	\author[\hspace{0.7cm}\centerline{
	}]{ K. Christy Rani$^{1}$, I. Dhivviyanandam$^2$ }
	\title[\centerline{  \hspace{0.5cm}}]{Detour Monophonic Vertex Cover Pebbling Number $(DMVCPN)$ of Some Standard Graphs}

	\thanks{\noindent $^1$ Department of Mathematics, Bon Secours College for Women (Autonomous), \newline Thanjavur-613006,Tamil Nadu, India.\\
		\indent \,\,\, e-mail: christy.agnes@gmail.com; https://orcid.org/0000-0002-2837-9388.\\
		\indent $^2$  Department of Mathematics,  North Bengal St. Xavier's College, Rajganj-735134,  West Bengal, India.\\
		\indent \,\,\, e-mail:  divyanasj@gmail.com; https://orcid.org/0000-0002-3805-6638.\\
	 }
		%\indent $^3$ Reg. No : 20211282091003, Research Scholar, Center: PG and Research Department of Mathematics, St. Xavier's College (Autonomous), Palayamkottai-627002,  Manonmaniam Sundaranar University, Abisekapatti-627012, Tamilnadu, India.\\
		%\indent \,\,\, e-mail: divyanasj@gmail.com;  https://orcid.org/0000-0002-3805-6638.\\
		%\indent \S \, Manuscript received: Month Day, Year; accepted: Month Day, Year. \\
		%\indent \,\,\, TWMS Journal of Applied and Engineering Mathematics, Vol.xx,
		%No.xx; \copyright\ I\c s\i k University, Depart-\\ 
		%\indent \,\,\, ment of Mathematics, 20xx; all rights reserved.\\
		%}
	
	%
	\begin{abstract}

Let $G$ be a connected graph with vertex set $V(G)$ and edge set $E(G)$. Pebbling shift is a deletion of two pebbles from a vertex and a placement of one pebble to a neighbouring vertex. The vertex cover set, $D_{vc}$ for graph $G$ is the subset  of $V(G)$ such that every edge in $G$ has minimum one end in  $D_{vc}$. A detour monophonic path is considered to be a longest chordless path between two non adjacent vertices $x$ and $y$.  A detour monophonic vertex cover pebbling number, ${\mu}_{vc}(G),$ is a minimum number of pebbles require to cover all the vertices of the vertex cover set of $G $ with at least one pebble each on them after the transformation of pebbles by using detour monophonic paths. We determine the detour monophonic vertex cover pebbling number (DMVCPN) of Cycle, Path,  Fan and Wheel graphs.   \\

		

	\end{abstract}
	\maketitle 
	\bigskip
	\bigskip
	%
	%%%%%%%%%%%%%%%%%%%%%%%%%%%%%%%%%%%%%%%%%%%%%%%%%%%%%%%%%%%%%%%%%%%%%%%%%%%%%%
	\section{\bf Introduction}
Graph pebbling is an interesting area emerged a few decades ago and enhanced Graph theory. Graph pebbling has produced results in  Combinatorial Group Theory and  Combinatorial Number Theory \cite{3}, \cite{4}.
	It is widely applied in network optimization model for transportation and for transmitting information through its medium to military troops, etc. 
		A pebbling move is an  extraction of  two pebbles from one vertex and placing one pebble on the adjacent vertex and eliminating the other pebble.  A vertex cover of a graph $G$ is studied in \cite{7},  which defines the vertex cover as a  set $D_{vc} \subset V(G)$ such that for  each edge $uv \in E(G)$, at least one of $ u$ or $v$ is in $D_{vc}$. 	A detour monophonic path from $x$ to $y$ is the  longest path with no chords studied in   \cite{8}.In this article we define the detour monophonic vertex cover pebbling number, ${\mu}_{vc}(G)$, which is a minimum number of pebbles require to cover all the vertices of the vertex cover set $D_{vc}$ of $G$ and the corresponding edges  of $G$ with at least one pebble each on them after the transformation of pebbles by using detour monophonic paths. We also investigate the detour monophonic vertex cover pebbling number of standard graphs. To prove the worst condition, we use the stacking theorem. It states that when the initial configuration $D$ is placed on a single vertex $v$ such that the distance of $v = {\dot{s}}$ is  maximum and hence ${\dot{s}} =  \sum_{u \in V{G}}{2^{dist(u,v)}}$. Applying this for every vertex $v\in V(G)$, we determine $\mu_{vc}(G)$. We refer  to \cite{1} and \cite{2} for basic terminologies of graph theory. Throughout this paper, we refer  $G$  to be a simple connected graph. \\

   	\begin{definition} \cite{8} 
 		In a graph $G$, the smallest number $\mu (G, x)$ for $x \in G$  such that minimum single  pebble can  be shifted to $x$ using a monophonic path by the successive pebbling shifts to place $\mu (G, x)$ pebbles  on the vertices of $G$. The maximum $\mu (G, x)$ over  $v \in V(G)$ is the monophonic pebbling number denoted by $\mu (G)$.
 	  \end{definition}
 	 	    
	\begin{definition}  \cite{6} 
	Any vertex $v \in V(G)$ is called a source vertex when $D(u, v)$ is maximum.
  \end{definition}

%	\begin{theorem} 
%	\cite{5}
%	For the path $P_n $, $\gamma_{ \mu}(P_n)$   =  $2^{n} \ - \ 1$.
   %\end{theorem}
	
	\begin{notation} 
		In this article we use 
		\item a. $D_{vc}$ for the vertex cover set, 
		\item b. ${\dot{s}}$ to denote the source vertex,
		\item c. $D_m$ for the detour monophonic distance and
		\item d. $\mu_{vc} (G) $ to denote the detour monophonic vertex cover pebbling number - $DMVCPN$.
		\end{notation}

\section{\bf Motivation}

  Motivated by the concepts of  the cover pebbling number \cite{3a} and detour monophonic number \cite{8} and monophonic pebbling number \cite{5} with vast applicatoins, we have discussed the detour monophonic vertex cover pebbling number for some standard graphs such as  $C_n (n \geq 3), P_n, F_n  $ and  $W_n$.  Detour monophonic vertex cover pebbling is applied in  deciding the equal distribution of goods on every customer by using the detour monophonic path. This is also applied in the network transmission of the information from one node to the other.\\
  
  %%%%%%%%%%%%%%%%%%%%%%%%%%%%%%%%%%%%%%%%%%%%%%%%%%%%%%%

% odd path  88888888888888888888888888888888888888888888
  \begin{theorem} \label{thm1}
 	Let $P_n$ be a path of order $n \ (n \ is \ odd  \ and \ n \geq 3)$. \\ Then 	$ \mu_{vc}(P_n) =
 	\sum_{i = 0}^{ \lfloor {\frac{n}{2}} \rfloor -1} 2^{2(i)+1}$. 
 		
	\end{theorem}

   \begin{proof}

Let $P_n$ be a path of order $n$, where $n$ is odd. Let
 $D_{vc}$ be the  vertex cover set. Then, $D_{vc}= \{v_2, v_4, v_6, \cdots , v_{n-1}  \}. $
  Let the initial configuration be	$ \sum_{i = 0}^{ \lfloor {\frac{n}{2}} \rfloor -1} 2^{2(i)+1}-1$ pebbles. Without loss of generality let ${\dot{s}}=v_1$. The detour monophonic distance from $v_1$ to any other vertex  is at most $n-2$. In order to cover the farthest vertex $v_{n-1}$ using the detour monophonic path we  require $2^{n-2}$ pebbles. Similarly, to cover the remaining vertices of  $D_{vc}$,  
 we require the following pebble distributions.  \\
 $2^1+ 2^{3}+ 2^{5}+ 2^{7}+\cdots+ 2^{n-2}$. That  sums upto $ \sum_{i = 0}^{ \lfloor {\frac{n}{2}} \rfloor -1} 2^{2(i)+1}$. But our initial configuration of pebbles is $ \sum_{i = 0}^{ \lfloor {\frac{n}{2}} \rfloor -1} 2^{2(i)+1}-1$ which is a contradiction.
  Therefore,\\
 
 $\mu_{vc}(P_n) \geq  \sum_{i = 0}^{ \lfloor {\frac{n}{2}} \rfloor -1} 2^{2(i)+1}$.\\
  
Now we prove the sufficient condition by using $\sum_{i = 0}^{ \lfloor {\frac{n}{2}} \rfloor -1} 2^{2(i)+1}$ pebbles on the vertices of $P_n $.\\

\noindent{\bf Case 1:} 
Let  ${\dot{s}}= v_1 \ \text{or} \  v_n $. \\
 Without loss of generality, let $v_n$ be the source vertex. Then $D*_{vc}=\{ v_{n-1}, v_{n-3}, v_{n-5},\cdots, v_2  \}$. 
 
\begin{table}[h] %model
	\centering 
	\begin{tabular}{|c|c|c|c|c|c|c|c|} \hline
		&$v_2$	&$v_{4}$	&$v_{6}$	&$\cdots$ &$v_{n-5}$	&$v_{n-3}$ &$v_{n-1}$ 
		\\ \hline
		$v_1$ &$1$ &$3$ & $5$ & $\cdots$ &$n-6$ &$n-4$ &$n-2$   \\ \hline
		
		$v_n$ &$n-2$ &$n-4$ &$n-6$  & $\cdots$ & $5$ &$3$ &$1$ \\ \hline
	\end{tabular}
	\vspace{0.2cm}
\caption{Detour monophonic distance from end vertices to  $V(D*_{vc})$.}
\label{tab1}
\end{table} 
	In order to cover the vertices of $P_n$, it suffices to cover the vertices of   vertex cover set $V(D*_{vc})$. 
 From Table \ref{tab1}, we require $2^{n-2}$, $2^{n-4}$ and $2^{n-6}$ pebbles from $v_n$ to cover $v_2$, $v_{4}$ and $v_{6}$ respectively. Proceeding in the same way we use $2^5, 2^3$ and $2$ pebbles to cover the vertices $v_{n-5}$, $v_{n-3}$ and $v_{n-1}$. Thus the following series of pebbles guarantees covering all the vertices of $D*_{vc}$.\\ 
 $\mu_{vc} (G, v_2) + \mu_{vc} (G, v_{4}) + \mu_{vc} (G, v_6) + \cdots + \mu_{vc} (G,v_{ n-5}) + \mu_{vc} (G,v_{ n-3}) + \mu_{vc} (G,v_{ n-1}) $\\
 $=2^{n-2}+2^{n-4}+2^{n-6}+ \cdots +  2^{5} + 2^{3} + 2^{1}$. Thus the number of pebbles used to cover $V(D*_{vc})$ is,\\
$ \sum_{i = 0}^{ \lfloor {\frac{n}{2}} \rfloor -1} 2^{2(i)+1}$. \\

       \noindent{\bf Case 2:} 
 When the middle vertex is the source vertex.   \\
Let  ${\dot{s}= v_{i},  i={\lfloor \frac{n}{2} \rfloor} +1} $.\\
\noindent{\bf Subcase 2.1:} 
If
  $n \equiv 1 (mod ~4)$ and then
 $D^1_{vc}= \{v_2, v_4, v_6, \cdots $,
 $v_{\lfloor{\frac{n}{2}\rfloor}-2}$	, $v_{\lfloor{\frac{n}{2}\rfloor}}$ ,  $v_{\lfloor{\frac{n}{2}\rfloor}+2}$ , $v_{\lfloor{\frac{n}{2}\rfloor}+4}, \cdots , v_{n-1}  \}.$ 
\begin{table}[h]
	\centering 
		\small
	\begin{tabular}{|c|c|c|c|c|c|c|c|c|c|c|c|c|} \hline
		&$v_2$	&$v_{4}$	&$v_{6}$ 	&$\cdots$ &$v_{\lfloor{\frac{n}{2}\rfloor}-2}$	&$v_{\lfloor{\frac{n}{2}\rfloor}}$ &$v_{\lfloor{\frac{n}{2}\rfloor}+2}$ &$v_{\lfloor{\frac{n}{2}\rfloor}+4}$ & $\cdots$ &$v_{n-3}$ & $v_{n-1}$ 
		\\ \hline
		$v_{ \lfloor {\frac{n}{2}} \rfloor +1}$ & ${\lfloor{\frac{n}{2}\rfloor}-1}$ & ${\lfloor{\frac{n}{2}\rfloor}-3}$ & ${\lfloor{\frac{n}{2}\rfloor}-5}$ 	&$\cdots$ & $3$ & $1$ & $1$ & $3$	&$\cdots$  & ${\lfloor{\frac{n}{2}\rfloor}-3}$& ${\lfloor{\frac{n}{2}\rfloor}-1}$  	\\ \hline
			\end{tabular}
	\vspace{0.2cm}
	\caption{Detour monophonic distance from the middle vertex of $ P_{n}$ \  to \ $V(D^1_{vc})$, where $ n \equiv 1 (mod ~4)$ .}
	\label{tab2}
	\end{table} 

 From Table
\ref{tab2}, we use $2^{\lfloor{\frac{n}{2}\rfloor}-1} $, $2^{\lfloor{\frac{n}{2}\rfloor}-3} $  and $2^{\lfloor{\frac{n}{2}\rfloor}-5} $ pebbles from $v_{\lfloor \frac{n}{2} \rfloor} +1$ to cover  $v_2$, $v_4$ and $v_6$. We continue in this manner and cover the vertices of $D^1_{vc}$ with the following pebble distribution. \\
$ \mu_{vc} (G, v_2) + \mu_{vc} (G, v_{4}) + \mu_{vc} (G, v_6) + \cdots+ \mu_{vc} (G, v_{\lfloor{\frac{n}{2}\rfloor}-2})+\mu_{vc} (G, v_{\lfloor{\frac{n}{2}\rfloor}})+\mu_{vc} (G, v_{\lfloor{\frac{n}{2}\rfloor}+2})+\mu_{vc} (G,v_{\lfloor{\frac{n}{2}\rfloor}+4})+ \cdots  + \mu_{vc} (G,v_{ n-3}) + \mu_{vc} (G,v_{ n-1}) $.\\
$=2^{{\lfloor{\frac{n}{2}\rfloor}-1}}+2^{\lfloor{\frac{n}{2}\rfloor}-3}+2^{\lfloor{\frac{n}{2}\rfloor}-5}   + \cdots +  2^{3} + 2^{1} + 2^{1}+ 2^{3}+ \cdots +  2^{\lfloor{\frac{n}{2}\rfloor}-3}+  2^{\lfloor{\frac{n}{2}\rfloor}-1 }$ \\
$= 2(2^{1} + 2^{3} + 2^{5}+ 2^{7}+ \cdots + 2^{\lfloor {\frac{n}{2}} \rfloor -1}$). \\
 Thus in order to cover the vertices of $D^1_{vc}$ we used \\
$2\big(\sum_{j = 1} ^{\lfloor {\frac{n}{4}} \rfloor } 2^{2(j)-1} \big)$ pebbles 
 %where  $n \equiv 1 (mod ~4) \text{\ and }  {\dot{s}= v_{ {\lfloor \frac{n}{2} \rfloor}  }}$ 
 which is less than the total number of pebbles available. \\

\noindent{\bf Subcase 2.2:} 
%Let ${\dot{s}}= v_{ {\lfloor \frac{n}{2} \rfloor}+1}$
Let $ n \equiv 3 (mod ~4)$. Without loss of generality $D^2_{vc}= \{v_2, v_4, v_6, \cdots$ , $v_{\lfloor{\frac{n}{2}\rfloor}-1}$	, $v_{\lfloor{\frac{n}{2}\rfloor}+1}$ , $v_{\lfloor{\frac{n}{2}\rfloor}+3} , \cdots , v_{n-1}  \} $.
\begin{table}[h]
	\centering 

	\begin{tabular}{|c|c|c|c|c|c|c|c|c|c|c|c|} \hline
		&$v_2$	&$v_{4}$	&$v_{6}$ 	&$\cdots$ &$v_{\lfloor{\frac{n}{2}\rfloor}-1}$	&$v_{\lfloor{\frac{n}{2}\rfloor}+1}$ &$v_{\lfloor{\frac{n}{2}\rfloor}+3}$ & $\cdots$  &$v_{n-3}$ & $v_{n-1}$ 
		\\ \hline 
		$v_{ \lfloor {\frac{n}{2}} \rfloor +1}$ & ${\lfloor{\frac{n}{2}\rfloor}-1}$ & ${\lfloor{\frac{n}{2}\rfloor}-3}$ & ${\lfloor{\frac{n}{2}\rfloor}-5}$ 	&$\cdots$ & $2$ & $0$ & $2$	&$\cdots$  &  ${\lfloor{\frac{n}{2}\rfloor}-3}$& ${\lfloor{\frac{n}{2}\rfloor}-1}$  	\\ \hline
		\end{tabular}
	\vspace{0.2cm}
	\caption{Detour monophonic distance from the middle vertex of $ P_{n}$ \  to \ $V(D^2_{vc})$, where $ n \equiv 3 (mod ~4)$ .}
	\label{tab3}
\end{table} 

From Table 
\ref{tab3}, we use  $2^{\lfloor{\frac{n}{2}\rfloor}-1} $, $2^{\lfloor{\frac{n}{2}\rfloor}-3} $  and $2^{\lfloor{\frac{n}{2}\rfloor}-5} $ pebbles to cover $v_2$, $v_4$ and $v_6$. Similarly we cover the vertices of $D^2_{vc}$ with the following pebble distributions. \\
$ \mu_{vc} (G, v_2) + \mu_{vc} (G, v_{4}) + \mu_{vc} (G, v_6) + \cdots+\mu_{vc} (G, v_{\lfloor{\frac{n}{2}\rfloor}-1})+\mu_{vc} (G, v_{\lfloor{\frac{n}{2}\rfloor}+1})+\mu_{vc} (G, v_{\lfloor{\frac{n}{2}\rfloor}+3})+ \cdots  + \mu_{vc} (G,v_{ n-3}) + \mu_{vc} (G,v_{ n-1}) $.\\
$=2^{{\lfloor{\frac{n}{2}\rfloor}-1}}+2^{\lfloor{\frac{n}{2}\rfloor}-3}+2^{\lfloor{\frac{n}{2}\rfloor}-5}  + \cdots +  2^{2} + 2^{0} + 2^{2}+  \cdots +  2^{\lfloor{\frac{n}{2}\rfloor}-3}+  2^{\lfloor{\frac{n}{2}\rfloor}-1 }$ \\
$=2^{0} + 2( 2^{2} + 2^{4}+ 2^{6}+ \cdots + 2^{\lfloor {\frac{n}{2}} \rfloor -1})$. \\

Thus,  we used $  
 2\big(\sum_{k = 1 } ^{\lfloor {\frac{n}{4}} \rfloor } 2^{k} \big)+1$ pebbles, which is less than $\sum_{i = 0}^{ \lfloor {\frac{n}{2}} \rfloor -1} 2^{2(i)+1}$.  \\

\noindent{\bf Case 3:} 
Consider ${\dot{s}}$ is an internal vertex.
Let  ${\dot{s}= v_{i} }$, where $2 \leq i  \leq n-1$,  $i \neq {{\lfloor {\frac{n}{2} } \rfloor}+1} $. \\

\noindent{\bf Subcase 3.1:} 
When $i$   is even.\\
Let $ V(D^3_{vc})= \{v_2$,	$v_{4}$	, $v_{6}$ 	, $\cdots$ , $v_{i-2}$	, $v_{i}$	, $v_{i+2}$ ,  $\cdots$ , $v_{n-3}$ ,  $v_{n-1}\} $.\\
\begin{table}[h]
	\centering 
	\begin{tabular}{|c|c|c|c|c|c|c|c|c|c|c|c|} \hline
		&$v_2$	&$v_{4}$	&$v_{6}$ 	&$\cdots$ 	&$v_{i-2}$	 &$v_{i}$ & $v_{i+2}$  &$\cdots$ &$v_{n-3}$ & $v_{n-1}$ 
		\\ \hline
		%	$v_1$ &$1$ &$3$ &$5$ &$7$ &$9$ 	&$\cdots$ &$v_{n-8}$ & $v_{n-6}$ &$v_{n-4}$ & $v_{n-2}$		\\ \hline 
		$v_i$ &$|i-2|$ &$|i-4|$  &$|i-6|$ 	&$\cdots$ &${2}$ & ${0}$ &${2}$ &$\cdots$ & ${n-(i+3)}$	& ${n-(i+1)}$	\\ \hline 
		\end{tabular}
	\vspace{0.2cm}
	\caption{Detour monophonic distance from the internal vertex of  $P_n$
		to $ V(D^3_{vc})$.}
	\label{tab4}
    \end{table} 

When $i=2$, we use  $2^{0} $, $2^{2} $, $2^{6}, \cdots, 2^{n-5}$ and $2^{n-3}$ pebbles to cover  $v_2$, $v_4$, $v_6, \cdots$, $v_{n-3}$ and $v_{n-5}$ respectively. when $i=4$, we use $2^{2} $, $2^{0} $,    $2^{2}, \cdots, 2^{n-7}$ and $2^{n-5}$ pebbles to cover $v_2$, $v_4$,  $v_6, \cdots$, $v_{n-3}$ and $v_{n-5}$ respectively. Continuing this process, when $i$ is even, we cover the vertices of $D^3_{vc}$   with the following pebble distribution. \\
$ \mu_{vc} (G, v_2) + \mu_{vc} (G, v_{4}) + \mu_{vc} (G, v_6) + \cdots +   \mu_{vc} (G,v_{i-2})  + \mu_{vc} (G,v_{ i}) +  \mu_{vc} (G,v_{ i+2}) + \cdots+  \mu_{vc} (G,v_{ n-3})  + \mu_{vc} (G,v_{ n-1}) $.\\
 From Table \ref{tab4}, we can sum up as follows.\\
$  \sum_{l = 1 } ^{ {\frac{i}{2}}  -1} 2^{2(l)} +  \sum_{m = 0} ^{ {\frac{{n}- (i+1)}{2}} } 2^{2(m)} $.\\

\noindent{\bf Subcase 3.2:} 
When $i$   is odd.\\
Consider $ V(D^4_{vc})= \{v_2$,	$v_{4}$	, $v_{6}$ 	, $\cdots$ , $v_{i-3}$	, $v_{i-1}$	, $v_{i+1}$ ,  $v_{i+3}$  , $\cdots$ , $v_{n-3}$ ,  $v_{n-1}\} $.
\begin{table}[h]
	\centering 
	\begin{tabular}{|c|c|c|c|c|c|c|c|c|c|c|c|c|} \hline
		&$v_2$	&$v_{4}$	&$v_{6}$ 	&$\cdots$ &$v_{i-3}$	&$v_{i-1}$	 &$v_{i+1}$ & $v_{i+3}$  &$\cdots$ &$v_{n-3}$ & $v_{n-1}$ 
		\\ \hline
		%	$v_1$ &$1$ &$3$ &$5$ &$7$ &$9$ 	&$\cdots$ &$v_{n-8}$ & $v_{n-6}$ &$v_{n-4}$ & $v_{n-2}$		\\ \hline 
		$v_i$ &$|i-2|$ &$i-4$  &$i-6$ 	&$\cdots$ &${3}$ &${1}$ & ${1}$ &${3}$ &$\cdots$ & ${n-(i+3)}$	& ${n-(i+1)}$	\\ \hline 
		
	\end{tabular}
	\vspace{0.2cm}
	\caption{Detour monophonic distance from the internal vertices of  $P_n$
		to $ V(D^4_{vc})$.}
	\label{tab5}
\end{table} 

From Table \ref{tab5}, when $i$ is odd,
we cover the vertices of $D^4_{vc}$ with the following pebble distribution. \\
$ \mu_{vc} (G, v_{2}) + \mu_{vc} (G, v_{4}) + \mu_{vc} (G, v_6) + \cdots + \mu_{vc} (G,v_{ n-3}) + \mu_{vc} (G,v_{i-1})  + \mu_{vc} (G,v_{ i+1}) +  \mu_{vc} (G,v_{ i+3}) + \cdots + \mu_{vc} (G,v_{ n-3})  + \mu_{vc} (G,v_{ n-1})$.\\
$= 2^{1}+ 2^{3} + \cdots + 2^{i-4}+ 2^{i-2}+  2^{1} + 2^{3} + \cdots + 2^{n-(i+3)}+ 2^{{n-(i+1)}}$.
Thus, we have utilized \\
$ \sum_{l = 1 } ^{\lfloor {\frac{i}{2}} \rfloor } 2^{2(l)-1} +
\sum_{m = 0 } ^{\lfloor {\frac{{n}- (i+1)}{2}} \rfloor } 2^{2(m)+1} $ pebbles, 
 which is less than the total number of pebbles available.
 Hence,  \\ $\mu_{vc}(P_n) =  \sum_{i = 0}^{ \lfloor {\frac{n}{2}} \rfloor -1} 2^{2(i)+1}$.\\
\end{proof}
%**even path***************************************************

  \begin{theorem} \label{thm2}
	Let $P_n$ be a path of order $n (n \ is \ even  \ and \ n \geq 2)$. \\ 
Then 	$ \mu_{vc}(P_n) =\sum_{i = 0}^{ {\frac{n}{2}}  -1} 2^{2(i)}$.
	\end{theorem}

    \begin{proof}
	
	Let $P_n$ be a path of order $n$, where $n$ is even. Let 
	$D_{vc}$ be the  vertex cover set. Then, $D_{vc}= \{v_1, v_3, v_5, \cdots , v_{n-1}  \}.$ Let  the initial configuration be  $\sum_{i = 0}^{ {\frac{n}{2}} -1} 2^{2(i)}-1$ and  ${\dot{s}}= v_1$. The detour monophonic distance from $v_1$ to any other vertex of $P_n$ is at most $n-2$. Then using  $2^{n-2}$ pebbles we could cover $v_{n-1}$. Similarly to cover $v_{n-3}$, we require $2^{n-4}$ pebbles. In a similar way  we can  cover the remaining vertices of  $D_{vc}$  
	 using the following pebble distributions.  \\
	$2^{2}+ 2^{4}+ 2^{6}+\cdots+ 2^{n-4}+2^{n-2}$ that sums upto $ \sum_{i = 0}^{ {\frac{n}{2}} -1} 2^{2(i)}-1$. Now, there exists a vertex  $v_1$ uncovered. Thus we fail to cover all the vertices of $D_{vc}$. Hence, \\
			$\mu_{vc}(P_n) \geq  \sum_{i = 0}^{  {\frac{n}{2}}  -1} 2^{2(i)}$.\\

  Now we prove the sufficient condition with $\sum_{i = 0}^{  {\frac{n}{2}}  -1} 2^{2(i)}$ pebbles on the vertices of $P_n$.  \\

   \noindent{\bf Case 1:} 

Assume the end vertex to be the source vertex. \\
Let  ${\dot{s}}= v_i$, where $i \in \{1,n\}$.  \\
Without loss of generality let ${\dot{s}}= v_1$ and $D^1_{vc}=\{ v_1, v_3, v_5,\cdots,  v_{n-3}, v_{n-1} \}$. \\

\begin{table}[h] 
	\centering 
	\begin{tabular}{|c|c|c|c|c|c|c|c|c|c|c|} \hline
		&$v_1$	&$v_{3}$	&$v_{5}$ &$v_{7}$ &$v_{9}$	&$\cdots$ &$v_{n-7}$	 &$v_{n-5}$ &$v_{n-3}$	 &$v_{n-1}$ 
		\\ \hline
		$v_1$ &$0$ &$2$ & $4$  &$6$ &$8$ & $\cdots$  &$n-8$ &$n-6$  &$n-4$ &$n-2$   \\ \hline
		
		\end{tabular}
	\vspace{0.2cm}
	\caption{Detour monophonic distance from end vertex to  $V(D^1_{vc})$.}
	\label{tab6}
\end{table} 

We make use of $2^{0}$, $2^{2}$, $2^{4}$ pebbles to cover the vertices $v_1$, $v_{3}$ and $v_{5}$ respectively. Proceeding in the same way we use $2^{n-4} $ and $2^ {n-2}$ pebbles to cover the vertices  $v_{n-3}$ and $v_{n-1}$. Thus, from Table \ref{tab6} the following series of pebble distribution guarantees covering of all the vertices of $D^1_{vc}$. \\
$ \mu_{vc} (G, v_1) + \mu_{vc} (G, v_{3}) + \mu_{vc} (G, v_5) + \mu_{vc} (G, v_7) + \cdots +\mu_{vc} (G,v_{ n-7}) + \mu_{vc} (G,v_{ n-5}) + \mu_{vc} (G,v_{ n-3}) + \mu_{vc} (G,v_{ n-1}) $\\
$=2^{0} + 2^{2} + 2^{4}+  \cdots + 2^{n-2}$. \\
$=\sum_{i = 0}^{ {\frac{n}{2}}  -1} 2^{2(i)}$.\\
By symmetry, we can cover the vertices of $D^1_{vc}$, when the source vertex is $v_n$. \\

\noindent{\bf Case 2:} 

Let the middle vertex be the source vertex. Then there exist two subcases. \\
\noindent{\bf Subase 2.1:} 

 When $n$ is even and $  n \equiv 0 (mod ~4)$.\\
Assume  ${\dot{s}}= v_{i}$. \\
	If $i={ \frac{n}{2} }$, then $D^2_{vc} = \{v_1, v_3, v_5, \cdots,  v_{ \frac{n}{2}-1 }, v_{ \frac{n}{2}+1 }, \cdots, v_{n-3}, v_{n-1}\} $. \\
If $i={ \frac{n}{2}  +1 } $, then $D^2*_{vc} = \{v_2, v_4, v_6, \cdots,  v_{ \frac{n}{2}-1 }, v_{ \frac{n}{2}+1 }, \cdots,  v_{n}\} $. \\ Without loss of generality, consider the vertex cover set $D^2_{vc}$.

\begin{table}[h] 
	\centering 
	\begin{tabular}{|c|c|c|c|c|c|c|c|c|c|c|c|} \hline
		&$v_1$	&$v_{3}$	&$v_{5}$	&$\cdots$ &$v_{ \frac{n}{2}-3}$	 &$v_{ \frac{n}{2}-1}$ &$v_{ \frac{n}{2}+1}$	 &$v_{ \frac{n}{2}+3}$ &$\cdots$	&$v_{n-3}$	&$ v_{n-1}$ \\ \hline
				$v_i$ &$\frac{n}{2}-1$ &$\frac{n}{2}-3$ & $\frac{n}{2}-5$ & $\cdots$  &$3$ &$1$  &$1$ &$3$ & $\cdots$ &$\frac{n}{2}-3$ &$\frac{n}{2}-1$  \\ \hline
		
	\end{tabular}
	\vspace{0.2cm}
	\caption{Detour monophonic distance from middle vertex to  $V(D^2_{vc})$.}
	\label{tab7}
    \end{table} 

Now, from Table \ref{tab7} the following pebble distribution covers the vertices of  $D^2_{vc}$.

$2^{{{\frac{n}{2}}-1}}+2^{\frac{n}{2}-3}+2^{\frac{n}{2}-5}   + \cdots +  2^{3} + 2^{1} + 2^{1}+ 2^{3}+ \cdots +  2^{\frac{n}{2}-3}+  2^{\frac{n}{2}-1 }$ \\
$= 2(2^{1} + 2^{3} + 2^{5}+ \cdots + 2^{\frac{n}{2} -1})$. \\ 
Similarly, we can cover the vertices of  $D^{2*}_{vc}$.
Thus in this case to cover the vertices of the vertex cover set, we use   \\
$ 2\big(\sum_{j = 0 }^{\frac{n}{4} } 2^{2(j)-1} \big)$ pebbles, where  $n \equiv 0 (mod ~4)$.  \\

\noindent{\bf Subcase 2.2:}
 
When $n $ is even and  $  n \equiv 2 (mod ~4)$.\\ Let  $i={ \frac{n}{2} }$. Then  ${\dot{s}}= v_{ \frac{n}{2} }$.\\
As we consider the  odd middle vertex, the vertex cover set also includes all the odd vertices of $P_n$. Thus, $D^3_{vc}= \{ v_1, v_3, v_5, \cdots, v_{ \frac{n}{2} }, v_{ \frac{n}{2}+2 }, \cdots,  v_{n-1}\}$.
Let  $i={ \frac{n}{2}+1 }$ and  ${\dot{s}}= v_{ \frac{n}{2}+1 }$.  When we consider the  even  middle vertex, the vertex cover set  consists of  all the even vertices of $P_n$. Hence $D^{3*}_{vc} = \{v_2, v_4, v_6, \cdots,  v_{ \frac{n}{2}-1 },  v_{ \frac{n}{2}+1 }, v_{ \frac{n}{2}+3 }, \cdots,  v_{n}\} $. \\Without loss of generality, assume  the vertex cover set $D^{3*}_{vc}$.

\begin{table}[h] 
	\centering 
	\begin{tabular}{|c|c|c|c|c|c|c|c|c|c|c|c|c|} \hline
		&$v_2$	&$v_{4}$	&$v_{6}$	&$\cdots$  &$v_{ \frac{n}{2}-3}$ &$v_{ \frac{n}{2}-1}$	 &$v_{ \frac{n}{2}+1}$ &$v_{ \frac{n}{2}+3}$	 &$v_{ \frac{n}{2}+5}$ &$\cdots$	&$v_{n-2}$	&$ v_{n}$ \\ \hline
		$v_i$ &$\frac{n}{2}-1$ &$\frac{n}{2}-3$ & $\frac{n}{2}-5$ & $\cdots$  &$4$ &$2$  &$0$ &$2$ &$4$ & $\cdots$ &$\frac{n}{2}-3$ &$\frac{n}{2}-1$  \\ \hline
			
	\end{tabular}
\vspace{0.2cm}
\caption{Detour monophonic distance from middle vertex to  $V(D^{3*}_{vc})$.}
\label{tab8}
\end{table} 
Thus, from Table \ref{tab8} we cover the vertices of $D^{3*}_{vc}$ with the following pebble distribution. \\
$ 2^{{\frac{n}{2}}-1}+ 2^{{\frac{n}{2}}-3}+ 2^{{\frac{n}{2}}-5}+ \cdots +  2^{4} + 2^{2} 
+ 2^{0} + 2^{2} + 2^{4}+  \cdots + 2^{{\frac{n}{2}}-3} + 2^{{\frac{n}{2}}-1}$. \\
$ =2^{2{\lfloor{\frac{n}{4}\rfloor}}}+\cdots +  2^{4} + 2^{2} 
+ 2^{0} + 2^{2} + 2^{4}+  \cdots  + 2^{2{\lfloor{\frac{n}{4}\rfloor}}}$. \\
In this manner we can cover the vertices of $D^{3}_{vc}$.
Hence, the pebble distribution is summed up as follows. \\
$1+2\big(\sum_{k = 1} ^{\lfloor{\frac{n}{4}\rfloor}}  2^{2(k)} \big),$ where  $n \equiv 2 (mod ~4) \text{\ and }  {\dot{s}}$ is a middle vertex.  \\

\noindent{\bf Case 3:} 
Consider ${\dot{s}}$ to be an internal vertex.
Let  ${\dot{s}= v_{i} }$, where $2 \leq i  \leq n-1$,  $i \neq  {\frac{n}{2} }, {{\frac{n}{2} } +1} $. \\

\noindent{\bf Subase 3.1:} 

If  $v_{i}$   is even
and $i <{\frac{n}{2}} $. Then $D^4_{vc}  = \{v_1, v_3, v_5, \cdots, 
v_{i-1}, v_{i+1},  \cdots,    v_{n-1}\} $.
\begin{table}[h] %model
	\centering 
	\begin{tabular}{|c|c|c|c|c|c|c|c|c|c|c|} \hline
		&$v_1$	&$v_{3}$	&$v_{5}$  &$\cdots$ &$v_{i-3}$	 &$v_{i-1}$	 &$v_{i+1}$	 &$v_{i+3}$	&$\cdots$  	 &$v_{n-1}$ 
		\\ \hline
		$v_i$ &$i-1$ &$i-3$ & $i-5$   & $\cdots$  &$3$ &$1$  &$1$ &$3$ & $\cdots$ & $n-(i+1)$  \\ \hline
		
	\end{tabular}
	\vspace{0.2cm}
	\caption{Detour monophonic distance from 	$v_i$ to  $V(D^4_{vc})$.}
	\label{tab9}
\end{table}  \\ 
From Table \ref{tab9} to cover the vertices of $D^4_{vc}$, we require the following pebble series. The vertices $v_1$ and $v_{n-1}$ are covered with $2^{i-1}$ and $2^{n-(i+1)}$ pebbles. Thus the required pebble distribution results as follows.\\
 $ \mu_{vc} (G, v_1) + \mu_{vc} (G, v_3) +\cdots +  \mu_{vc} (G, v_{i-3}) + \mu_{vc} (G,v_{ i-1}) + \mu_{vc} (G,v_{ i+1})  + \mu_{vc} (G,v_{ i+3})   + \cdots + \mu_{vc} (G,v_{ n-1}) $.\\
 $=  2^{i-1}+ 2^{i-3}+\cdots +2^{3}+2^{1}+ 2^{1}+2^{3}+\cdots+ 2^{n-(i+1)} $\\
$ = \sum_{j = 1 } ^{\lfloor {\frac{i}{2}} \rfloor } 2^{2(j)-1} +  \sum_{k =1}^
		{\lfloor {\frac{n-(i+1)}{2}} \rfloor } 2^{2(k)-1} $.\\
		
When  $i > {\frac{n}{2}+1} $, we obtain the reversal result using the similar pebble distribution. Hence  the number of pebbles employed in the process of covering $D^4_{vc}$ is, \\
 $ \sum_{j =1}^
 {\lfloor {\frac{n-(i+1)}{2}} \rfloor } 2^{2(j)-1} +
 \sum_{k = 1 } ^{\lfloor {\frac{i}{2}} \rfloor } 2^{2(k)-1} 
 $.\\

 \noindent{\bf Subcase 3.2:} 
 
If  $v_{i}$   is odd
and $i < {\frac{n}{2}} $. Then $D^5_{vc}  = \{v_1, v_3, v_5, \cdots, 
v_{i}, v_{i+2},  \cdots,    v_{n-1}\} $. 
\begin{table}[h] %model
	\centering 
	\begin{tabular}{|c|c|c|c|c|c|c|c|c|c|c|c|} \hline
		&$v_1$	&$v_{3}$	&$v_{5}$  &$\cdots$ &$v_{i-4}$ &$v_{i-2}$	 &$v_{i}$	 &$v_{i+2}$	 &$v_{i+4}$	&$\cdots$  	 &$v_{n-1}$ 
		\\ \hline
		$v_i$ &$i-1$ &$i-3$ &$i-5$  & $\cdots$ &$4$ &$2$ &$0$  &$2$ &$4$ & $\cdots$ & $n-(i+1)$  \\ \hline
		
	\end{tabular}
	\vspace{0.2cm}
	\caption{Detour monophonic distance from 	$v_i$ to  $V(D^5_{vc})$.}
	\label{tab10}
\end{table}  \\ 
The vertex $v_1$ is covered with $2^{i-1}$ pebbles and $v_{n-1}$ is covered with  $2^{n-(i+1)}$ pebbles. Thus, from Table \ref{tab10} to cover the vertices of $D^5_{vc}$, we require the following pebble distributions.\\  
$\mu_{vc} (G, v_1) + \cdots + \mu_{vc} (G, v_{i-4}) +  \mu_{vc} (G, v_{i-2}) + \mu_{vc} (G,v_{ i}) + \mu_{vc} (G, v_{ i+2})  + \mu_{vc} (G,v_{i+4})  +  \cdots + \mu_{vc} (G,v_{ n-1}) $.\\
$=  2^{i-1}+ \cdots +2^{4}+2^{2}+2^{0}+ 2^{2}+2^{4}+\cdots+ 2^{n-(i+1)} $\\
$ = \sum_{l = 0 } ^{\lfloor {\frac{i}{2}} \rfloor } 2^{2(l)} +  \sum_{m =1}^
{\lfloor {\frac{n-(i+1)}{2}} \rfloor } 2^{2(m)} $.\\

If  $i > {\frac{n}{2}+1} $, we obtain the reversal result using the similar pebble distribution. Hence, \\
 $\sum_{l =1}^
 {\lfloor {\frac{n-(i+1)}{2}} \rfloor } 2^{2(l)} + \sum_{m = 0 } ^{\lfloor {\frac{i}{2}} \rfloor } 2^{2(m)}$.
\\
Thus in each case, we have used the  pebbles which sumps up less than the considered pebbles. Therefore, 	$\mu_{vc}(P_n) = \sum_{i = 0}^{  {\frac{n}{2}}  -1} 2^{2(i)}$.\\
\end{proof} 

%%%%%%%%%%%%%%%%%% fan %%%%%%%%%%%%%%%%%%%%%%%%

   \begin{theorem} \label{thm3}
	For Fan $F_n  (n \ is \ even), \  {\mu_{vc}}(F_n) =
	\sum_{i = 0}^{ {\frac{n}{2}}  -2} 2^{2(i)+1} +2$. 
	\end{theorem}

	\begin{proof} 
Let $F_n= P_{n-1}+v_0$ be a fan of order $n ( n $ \ is \ even) with $V(F_n)= \{v_0, v_1, v_2,  \cdots , v_{n-1}  \}$. \\
Let 
$D_{vc}$ be the  vertex cover set. Then, $D_{vc}= \{v_0, v_2, v_4, v_6, \cdots , v_{n-2}  \}. $ Consider  the following configuration of pebbles on the vertices of fan  $F_n$. \\
Let the initial configuration be $  \sum_{i = 0}^{ {\frac{n}{2}}  -2} 2^{2(i)+1} + 1$ pebbles.  Assume $v_1$ to be the source vertex. The detour monophonic distance from $v_1$ to any other vertex is at most $n-2$. Then using  $2^{n-3}$ pebbles we could place a pebble on $v_{n-2}$. Then we place $2^{n-5}$ pebbles on $v_{n-4}$. In a similar way, we distribute the pebbles on the vertices of $D_{vc}$.  Now, there remains only one pebble on $v_2$ leaving  $v_1$ pebbleless. Thus, we fail to cover all the vertices of $D_{vc}$. Therefore, \\
$\mu_{vc}(F_n) \geq  \sum_{i = 0}^{ {\frac{n}{2}}  -2} 2^{2(i)+1}+ 2$.\\
 Now we prove the sufficient condition by using $\sum_{i = 0}^{ {\frac{n}{2}}  -2} 2^{2(i)+1}+ 2$ pebbles on the vertices of $F_n $.\\

  \noindent{\bf Case 1:} 
 Let  ${\dot{s}}= v_1 \ \text{or} \  v_{n-1} $. \\
 Without loss of generality, let $v_n$ be the source vertex. Then $D^1_{vc}=\{ v_{0}, v_{2}, v_{4}, v_{6}, \cdots, \\ v_{n-6}, v_{n-4}, v_{n-2}   \}$.

 \begin{table}[h] 
 	\centering 
 	\begin{tabular}{|c|c|c|c|c|c|c|c|c|} \hline
 		&$v_0$	&$v_2$	&$v_{4}$	&$v_{6}$	&$\cdots$ &$v_{n-6}$	&$v_{n-4}$ &$v_{n-2}$ 
 		\\ \hline
 		$v_1$ &$1$  &$1$ &$3$ & $5$ & $\cdots$ &$n-7$ &$n-5$ &$n-3$   \\ \hline
 		
 		$v_{n-1}$ &$1$ &$n-3$ &$n-5$ &$n-7$  & $\cdots$ & $5$ &$3$ &$1$ \\ \hline
 	\end{tabular}
 	\vspace{0.2cm}
 	\caption{Detour monophonic distance from $v_{1}$ and $v_{n-1}\  to \ V(D^1_{vc})$.}
 	\label{tab11}
 \end{table} 
 From Table 
 \ref{tab11}, we use $2$ pebbles to cover $v_0$. Also, we ues  $2^{n-3}$, $2^{n-5}$, $2^{n-7}$ pebbles to cover $v_0$, $v_2$, $v_{4}$ and $v_{6}$ respectively. Proceeding in the same way we need $2^5, 2^3$ and $2$ pebbles to cover the vertices $v_{n-6}$, $v_{n-4}$ and $v_{n-2}$. Thus, the following series of pebbles guarantees covering all the vertices of $D^1_{vc}$.\\ 
 $ \mu_{vc} (G, v_0) + \mu_{vc} (G, v_2) + \mu_{vc} (G, v_{4}) + \mu_{vc} (G, v_6) + \cdots + \mu_{vc} (G,v_{ n-6}) + \mu_{vc} (G,v_{ n-4}) + \mu_{vc} (G,v_{ n-2}) $\\
 $=2^{1}+ 2^{n-3}+2^{n-5}+2^{n-7}+ \cdots +  2^{5} + 2^{3} + 2^{1}$.  Thus,\\
 	$  \sum_{i = 0}^{ {\frac{n}{2}}  -2} 2^{2(i)+1}+2$ pebbles are used to cover $D^1_{vc}$.  \\

 \noindent{\bf Case 2:} 
 
 Consider  ${\dot{s}= v_{i},  i={ \frac{n}{2} } } $.\\
 \noindent{\bf Subcase 2.1:} 
 If
 $n \equiv 2 (mod ~4)$ and then
 $D^2_{vc}= \{v_0, v_2, v_4, v_6, \cdots ,v_{\frac{n}{2}-3},	v_{\frac{n}{2}-1}, v_{\frac{n}{2}+1}, \\ v_{\frac{n}{2}+3},\cdots, v_{n-4}, v_{n-2}  \}.$ \\
 
 \begin{table}[h]
 	\centering 
 	\begin{tabular}{|c|c|c|c|c|c|c|c|c|c|c|c|c|c|} \hline
 		&$v_0$	&$v_2$	&$v_{4}$	&$v_{6}$ 	&$\cdots$ &$v_{\frac{n}{2}-3}$	&$v_{\frac{n}{2}-1}$ &$v_{\frac{n}{2}+1}$ &$v_{\frac{n}{2}+3}$ & $\cdots$  &$v_{n-4}$ & $v_{n-2}$ 
 		\\ \hline 
 		$v_{\frac{n}{2}}  $ &$1$ & ${\frac{n}{2}-2}$ & ${\frac{n}{2}-4}$ & ${\frac{n}{2}-6}$ 	&$\cdots$ & $3$ & $1$ & $1$ & $3$	&$\cdots$  & ${\frac{n}{2}-4}$  & ${\frac{n}{2}-2}$   	\\ \hline
 	 	\end{tabular}
 	\vspace{0.2cm}
 	\caption{Detour monophonic distance from $v_{\frac{n}{2} }$ of $ F_{n}$   to  $V(D^2_{vc})$, where $ n \equiv 2 (mod ~4)$ .}
 	\label{tab12}
 \end{table} 
 
 From Table
 \ref{tab12}, we use $2^{\frac{n}{2}-2} $, $2^{\frac{n}{2}-4} $   pebbles to cover  $v_2$, $v_{n-2}$ and $v_4$, $v_{n-4}$ respectively. Obviously. $v_0$ is covered with 2 pebbles.  We continue in this manner and cover the vertices of $D^2_{vc}$ with the following pebble distribution. \\
 
  $ \mu_{vc} (G, v_0) +\mu_{vc} (G, v_2) + \mu_{vc} (G, v_{4}) + \mu_{vc} (G, v_6) + \cdots+ \mu_{vc} (G, v_{\frac{n}{2}-3})\mu_{vc} (G, v_{\frac{n}{2}-1})+\mu_{vc} (G, v_{\frac{n}{2}+1})+\mu_{vc} (G, v_{\frac{n}{2}+3})+ \cdots  + \mu_{vc} (G,v_{ n-4}) + \mu_{vc} (G,v_{ n-2}) $.\\
  
  $=2^{1} + 2^{\frac{n}{2}-2} + 2^{\frac{n}{2}-4}+2^{\frac{n}{2}-6}   + \cdots +  2^{3} + 2^{1} + 2^{1}+ 2^{3}+ \cdots +  2^{\frac{n}{2}-4}+  2^{\frac{n}{2}-2}$. \\
  Thus in order to cover the vertices of $D^2_{vc}$ we used \\
 $2\big(\sum_{j = 1} ^{\lfloor {\frac{n}{4}} \rfloor } 2^{2(j)-1} \big) +2$ pebbles, 
 %where  $n \equiv 1 (mod ~4) \text{\ and }  {\dot{s}= v_{ {\lfloor \frac{n}{2} \rfloor}  }}$ 
 which is less than the total number of pebbles available. \\
 
  \noindent{\bf Subcase 2.2:} 
 %Let ${\dot{s}}= v_{ {\lfloor \frac{n}{2} \rfloor}+1}$
 Let $ n \equiv 0 (mod ~4)$. Without loss of generality $D^3_{vc}= \{v_0, v_2, v_4, v_6, \cdots ,\\ v_{\frac{n}{2}-2},	v_{\frac{n}{2}}, v_{\frac{n}{2}+2},   \cdots , v_{n-4}, v_{n-2}  \} $.\\
 \begin{table}[h]
 	 	\centering 
 	 	\begin{tabular}{|c|c|c|c|c|c|c|c|c|c|c|c|c|} \hline
 	&$v_0$	&$v_2$	&$v_{4}$	&$v_{6}$ 	&$\cdots$ &$v_{\frac{n}{2}-2}$	&$v_{\frac{n}{2}}$ &$v_{\frac{n}{2}+2}$ & $\cdots$  &$v_{n-4}$ & $v_{n-2}$ 
 		\\ \hline 
 		$v_{\frac{n}{2}}  $ &$1$ & ${\frac{n}{2}-2}$ & ${\frac{n}{2}-4}$ & ${\frac{n}{2}-6}$ 	&$\cdots$ & $2$ & $0$ & $2$	&$\cdots$  & ${\frac{n}{2}-4}$  & ${\frac{n}{2}-2}$   	\\ \hline
 		
 	\end{tabular}
 	\vspace{0.2cm}
 	\caption{Detour monophonic distance from the middle vertex of $ F_{n}$ \  to \ $V(D^3_{vc})$, where $ n \equiv 0 (mod ~4)$ .}
 	\label{tab13}
 \end{table} 
 
  From Table 
 \ref{tab13}, $v_0$ is covered with 2 pebbles.  We use  $2^{\frac{n}{2}-2} $, $2^{\frac{n}{2}-4} $  and $2^{\frac{n}{2}-6} $ pebbles to cover $v_2$, $v_4$ and $v_6$. Similarly we cover the vertices of $D^3_{vc}$ with the following pebble distributions. \\
 $ \mu_{vc} (G, v_0) +\mu_{vc} (G, v_2) + \mu_{vc} (G, v_{4}) + \mu_{vc} (G, v_6) + \cdots+\mu_{vc} (G, v_{\frac{n}{2}-2})+\mu_{vc} (G, v_{\frac{n}{2}})+\mu_{vc} (G, v_{\frac{n}{2}+2})+ \cdots  + \mu_{vc} (G,v_{ n-4}) + \mu_{vc} (G,v_{ n-2}) $.\\
 %$=2^{0} + 2^{{\frac{n}{2}-2}}+2^{\frac{n}{2}-4}+2^{\frac{n}{2}-6}  + \cdots +  2^{2} + 2^{0} + 2^{2}+  \cdots  +2^{\frac{n}{2}-4}+2^{\frac{n}{2}-2 } $ \\
 $=2^{0} + 2( 2^{2} + 2^{4}+ 2^{6}+ \cdots + 2^{ {\frac{n}{4}}  -1})+2^1$. \\
 
 Thus  we used $  
 2\big(\sum_{k = 1 } ^{{ {\frac{n}{4}}  }-1} 2^{k} \big)+3$, which is less than $\sum_{i = 0}^{ {\frac{n}{2}}  -2} 2^{2(i)+1}+2$.  \\
 
  \noindent{\bf Case 3:} 
 Consider ${\dot{s}}$ is an internal vertex.
 Let  ${\dot{s}= v_{i} }$, where $2 \leq i  \leq n-2$,  $i \neq {{\lfloor {\frac{n-1}{2} } \rfloor}+1} $. \\
 
 \noindent{\bf Subcase 3.1:} 
 When $i$   is even.\\
 Let $ V(D^4_{vc})= \{v_0, v_2$,	$v_{4}$	, $v_{6}$ 	, $\cdots$ , $v_{i-2}$	, $v_{i}$	, $v_{i+2}$ ,  $\cdots$ , $v_{n-4}$ ,  $v_{n-2}\} $.\\
 \begin{table}[h]
 	\centering 
 	\begin{tabular}{|c|c|c|c|c|c|c|c|c|c|c|c|} \hline
 &$	v_0$	&$v_2$	&$v_{4}$	&$v_{6}$ 	&$\cdots$ 	&$v_{i-2}$	 &$v_{i}$ & $v_{i+2}$  &$\cdots$ &$v_{n-4}$ & $v_{n-2}$ 
 		\\ \hline
 		%	$v_1$ &$1$ &$3$ &$5$ &$7$ &$9$ 	&$\cdots$ &$v_{n-8}$ & $v_{n-6}$ &$v_{n-4}$ & $v_{n-2}$		\\ \hline 
 		$v_i$ &${1}$ &$|i-2|$ &$|i-4|$  &$|i-6|$ 	&$\cdots$ &${2}$ & ${0}$ &${2}$ &$\cdots$ & ${n-(i+4)}$	& ${n-(i+2)}$	\\ \hline 
 		
 	\end{tabular}
 	\vspace{0.2cm}
 	\caption{Detour monophonic distance from $v_i$ of  $F_n$
 		to $ V(D^4_{vc})$.}
 	\label{tab14}
 \end{table} 
 From Table 
 \ref{tab14}, when $i=2$, we use  2, $2^{0} $, $2^{2} $, $2^{6}, \cdots, 2^{n-6}$ and $2^{n-4}$ pebbles to cover $v_0$, $v_2$, $v_4$, $v_6, \cdots$, $v_{n-4}$ and $v_{n-6}$ respectively. when $i=4$, we use $2^{0} $, $2^{2} $, $2^{0} $,    $2^{2}, \cdots, 2^{n-8}$ and $2^{n-5}$ pebbles to cover $v_2$, $v_4$,  $v_6, \cdots$, $v_{n-4}$ and $v_{n-6}$ respectively. Continuing this process, when $i$ is even, we cover the vertices of $D^4_{vc}$   with the following pebble distribution. \\
 $ \mu_{vc} (G, v_0) + \mu_{vc} (G, v_2) + \mu_{vc} (G, v_{4}) + \mu_{vc} (G, v_6) + \cdots +   \mu_{vc} (G,v_{i-2})  + \mu_{vc} (G,v_{ i}) +  \mu_{vc} (G,v_{ i+2}) + \cdots+  \mu_{vc} (G,v_{ n-4})  + \mu_{vc} (G,v_{ n-2}) $.\\
 $ = \sum_{l = 1 } ^{ {\frac{i}{2}}  -1} 2^{2(l)} +  \sum_{m = 0} ^{ {\frac{{n}- (i+2)}{2}} } 2^{2(m)}+2 $.\\
 
 \noindent{\bf Subcase 3.2:} 
 When $i$   is odd.\\
 Consider $ V(D^5_{vc})= \{v_0, v_2$,	$v_{4}$	, $v_{6}$ 	, $\cdots$ , $v_{i-3}$	, $v_{i-1}$	, $v_{i+1}$ ,  $v_{i+3}$  , $\cdots$ , $v_{n-4}$ ,  $v_{n-2}\} $.
 \begin{table}[h]
 	\centering 
 	\small
 	\begin{tabular}{|c|c|c|c|c|c|c|c|c|c|c|c|c|c|} \hline
 	&$v_0$	&$v_2$	&$v_{4}$	&$v_{6}$ 	&$\cdots$ &$v_{i-3}$	&$v_{i-1}$	 &$v_{i+1}$ & $v_{i+3}$  &$\cdots$ &$v_{n-4}$ & $v_{n-2}$ 
 		\\ \hline
 		%	$v_1$ &$1$ &$3$ &$5$ &$7$ &$9$ 	&$\cdots$ &$v_{n-8}$ & $v_{n-6}$ &$v_{n-4}$ & $v_{n-2}$		\\ \hline 
 		$v_i$ &$1$ &$|i-2|$ &$|i-4|$  &$|i-6|$ 	&$\cdots$ &${3}$ &${1}$ & ${1}$ &${3}$ &$\cdots$ & ${n-(i+4)}$	& ${n-(i+2)}$	\\ \hline 
 		
 	\end{tabular}
 	\vspace{0.2cm}
 	\caption{Detour monophonic distance from $v_i$ of  $F_n$
 		to $ V(D^5_{vc})$.}
 	\label{tab15}
 \end{table} 
 
 When $i$ is odd,
 we cover the vertices of $D^5_{vc}$ with the following pebble distribution. \\
 $ \mu_{vc} (G, v_{0}) +\mu_{vc} (G, v_{i-1}) + \mu_{vc} (G, v_{i-3}) + \mu_{vc} (G, v_{4}) + \mu_{vc} (G, v_6) + \cdots + \mu_{vc} (G,v_{ n-6}) + \mu_{vc} (G,v_{i-1})  + \mu_{vc} (G,v_{ i+1}) +  \mu_{vc} (G,v_{ i+3}) + \cdots + \mu_{vc} (G,v_{ n-4})  + \mu_{vc} (G,v_{ n-2})$.\\
 $=2^{1}+ 2^{1}+ 2^{3} + \cdots + 2^{i-4}+ 2^{i-2}+  2^{1} + 2^{3} + \cdots + 2^{n-(i+4)}+ 2^{{n-(i+2)}}$.
 Thus, we have ended up with \\
 $ \sum_{l = 1 } ^{\lfloor {\frac{i}{2}} \rfloor } 2^{2(l)-1} +
 \sum_{m = 0 } ^{\lfloor {\frac{{n}- (i+2)}{2}} \rfloor } 2^{2(m)+1}+2 $ pebbles 
 which is less than the total number of pebbles available.\\
 
  	\noindent{\bf Case 4:}  
 
 Consider the hub vertex as the source  vertex.
 % check 
 Then ${\dot{s}= v_{0} }$ and \\$D^6_{vc}=\{ v_{n-2}, v_{n-4}, v_{n-6},\cdots, v_2, v_0   \}$. \\
 The detour monophonic distance from any other vertex of $F_n$ to $v_0$ is 1. Thus in order to cover the vertices of $D^6_{vc}$, we have the  pebble distribution that sums up as follows.
 $\mu_{vc} (G, v_0) + \mu_{vc} (G, v_2) + \mu_{vc} (G, v_{4}) + \cdots + \mu_{vc} (G,v_{ n-6}) + \mu_{vc} (G,v_{ n-4}) + \mu_{vc} (G,v_{ n-2}) $\\
 $= 2^0+ 2^1+2^1+\cdots +2^1+2^1+2^1$\\
 $ 2 { \lfloor {\frac{n-1}{2}} \rfloor } + 1$ which is less than the considered number of pebbles.\\ 
   Hence,  \\ $\mu_{vc}(F_n) = \sum_{i = 0}^{ {\frac{n}{2}}  -2} 2^{2(i)+1}+2$, when $n$ is even.\\
 
 \end{proof}

\begin{theorem} \label{3a}
	For Fan $F_n   (n\ is \ odd ),  
	\mu_{vc}(F_n) =
				\sum_{j = 0}^{ \lfloor {\frac{n}{2}} \rfloor -1 } 2^{2(j)} +  2$.
		\end{theorem}

\begin{proof} 
The proof is similar to that of Theorems \ref{thm2} and \ref{thm3}.
\end{proof} 

%%%%%%%%%%%%%%%%%%%   cycle %%%%%%%%%%%%%%%%
\begin{theorem} \label{thm4}
	For  Cycle  $C_n  (n \geq 3), \\ {\mu_{vc}}(C_n)$ =
	$\begin{cases}
		2\big(\sum_{i = \frac{n}{4} + 1}^{\frac{n}{2} -1} 2^{2(i)} \big) + 2^{\frac{n}{2} } + 1, & \text{if } n \ is \ \text{even } \text{and }   n \equiv 0 (mod ~4)\\
		
		2\big(\sum_{j = { \lfloor \frac{n}{4} \rfloor } + 1}^{\frac{n}{2}-1} 2^{2(j)} \big) +  1, & \text{if } n \ is \ \text{even } \text{and } 
		\ n \equiv 2 (mod ~4)\\
		{\sum_{k = {\lfloor \frac {n}{2} \rfloor}+1}^{n-2} 2^{k}} + 3,  &  \text{if } n \ is \ \text{odd }.
		
	\end{cases}$
	
\end{theorem}
	
\begin{proof}
	Let $C_n$ be an even cycle of order $n (n \geq 3)$. \\
		
	\noindent{\bf Case 1:} 
	
	When $n \equiv 0 (mod ~4)$.\\
	Consider the detour monophonic vertex cover set $D^1_{vc} = \{ v_{i+ 2(j)} | i= 1,2,3,\cdots,n; j= 0,1,2,\cdots,{ \frac {n}{4} -1}, { \frac{n}{4}},  { \frac{n}{4}+1}, \cdots,  { \frac{n}{2}-2}, { \frac{n}{2}-1} \},$ the indices run over modulus $n$.  Thus $|D^1_{vc}|= \frac{n}{2}$.
	
	Let ${\dot{s}}  = v_1$ and  the initial configuration be	$ 2\big(\sum_{i = \frac{n}{4} + 1}^{\frac{n}{2} -1} 2^{2i} \big) + 2^{\frac{n}{2} } $. 
	In order to cover the vertices    $v_{(n-1)}$ we require  $2^ {n - 2}$ pebbles  and to cover  $v_{(n-3)}$ we require  $2^ {n - 4}$ pebbles respectively using the detour monophonic path. Similar process can be carried on to cover the remaining vertices of $D^1_{vc}$ using the following series of pebble distribution that sums up \\
	$
	2^{n-2}+ 2^{n-4}+ \cdots + 2^{\frac{n}{2}+2 }+ 2^{\frac{n}{2}}+2^{\frac{n}{2}+2 }+ \cdots + 2^{n-2}+2^{n-4} $\\
	$= 2( 2^{n-2}+ 2^{n-4}+\cdots +2^{{\frac{n}{2}+2 }} )+2^{\frac{n}{2}}$ \\
	$ = 2\big(\sum_{i = \frac{n}{4} + 1}^{\frac{n}{2} -1} 2^{2i} \big) + 2^{\frac{n}{2} }$.\\ 
	Thus we can cover the vertices $v_3$ and $v_5$ with $2^ {n - 2}$ and $2^ {n - 4}$ pebbles and we are left with no pebble to cover the source vertex $v_1$. Therefore, the detour monophonic vertex cover number ${\mu_{vc}}(C_n) \geq 2\big(\sum_{i = \frac{n}{4} + 1}^{\frac{n}{2} -1} 2^{2(i)} \big) + 2^{\frac{n}{2} } + 1$.  
	
	\begin{table}[h]
		\centering 
		\begin{tabular}{|c|c|c|c|c|c|c|c|c|c|c|} \hline
			&$v_i$	&$v_{i+2}$	&$v_{i+4}$	&$\cdots$ &$v_{{\frac{n}{2}+i-2}}$	&$v_{\frac{n}{2}+i}$ &$v_{{\frac{n}{2}+i+2}}$ & $\cdots$ &$v_{n+i-4}$ &$v_{n+i-2}$
			\\ \hline
			$v_i$ &$0$ &$n-2$ &$n-4$ &$\cdots$ &$\frac{n}{2}+2$ &$\frac{n}{2}$ &$\frac{n}{2}+2$  &$ \cdots $ &$n-4$ &$n-2$   \\ \hline
			
		\end{tabular}
		\vspace{0.2cm}
		\caption{Detour monophonic distance from $v_{i} \  to \ V(D^1_{vc})$.}
		\label{tab16}
	\end{table} 
	Obviously $v_i$ is covered with one pebble. From Table 
	\ref{tab16}, in order to cover the vertices   $v_{i+2}$ and $v_{n+i-2}$ we require  $2^ {n - 2}$ pebbles each and to cover  $v_{i+4}$ and $v_{n+i-4}$ we require  $2^ {n - 4}$ pebbles each. 
	Similarly  to cover $v_{{\frac{n}{2}+i-2}}$ and $v_{{\frac{n}{2}+i+2}}$, we require $ 2^ {\frac{n}{2} +2}$ pebbles. We need $2^ {\frac{n}{2}}$  
	pebbles to cover the vertex $v_{i+{\frac{n}{2}}}$. Thus in order to cover the vertices in $D^1_{vc}$ we distribute the pebbles as follows.
	$\mu_{vc} (G, v_i) + \mu_{vc} (G, v_{i+2}) +  \mu_{vc} (G, v_{i+4})+ \cdots + \mu_{vc}(G, v_{{\frac{n}{2}+i-2}} )+\mu_{vc} (G, v_{{\frac{n}{2}+i}}) + \mu_{vc} (G, v_{{\frac{n}{2}+i+2}} ) + \cdots +\mu_{vc} (G, v_{n+i-4})+\mu_{vc} (G, v_{n+i-2}) $. \\
	$ = 2^0+ 
	2^{n-2}+ 2^{n-4}+ \cdots + 2^{\frac{n}{2}+2 }+ 2^{\frac{n}{2}}+2^{\frac{n}{2}+2 }+ \cdots + 2^{n-4}+2^{n-2} $\\
	$=1 + 2( 2^{n-2}+ 2^{n-4}+\cdots +2^{{\frac{n}{2}+2 }} )+2^{\frac{n}{2}}$ \\
	$= 2\big(\sum_{i = \frac{n}{4} + 1}^{\frac{n}{2} -1} 2^{2i} \big) + 2^{\frac{n}{2} } + 1 $.\\
	Since all the vertices are of same detour monophonic distance, the number of pebbles required to place on any $v_i \in V(D^1{vc})$ is $2\big(\sum_{i = \frac{n}{4} + 1}^{\frac{n}{2} -1} 2^{2(i)} \big) + 2^{\frac{n}{2} } + 1 $.\\
		Thus,   $\mu_{vc} (C_n) =  2\big(\sum_{i = \frac{n}{4} + 1}^{\frac{n}{2} -1} 2^{2i} \big) + 2^{\frac{n}{2} } + 1 $.\\ 
		
	\noindent{\bf Case 2:} 
	
	When $n$ is even and $n \equiv 2 (mod ~4)$, let $D^2_{vc} = \{ v_{i+ 2j} |i= 1,2,3,\cdots,n; j= 0,1,2,\cdots, \lfloor { \frac{n}{4}-1}	\rfloor,  \lfloor {\frac{n}{4}} \rfloor,  \lfloor { \frac{n}{4}+1} \rfloor, \cdots,  { \frac{n}{2}-2}, { \frac{n}{2}-1} \}$ be the vertex cover set, where the indices run over modulus $n$. Thus, $|D^2_{vc}|= \frac{n}{2}$. 
	Now consider the configuration  which encages $ 2\big(\sum_{j = {\lfloor {\frac{n}{4}} \rfloor} + 1}^{\frac{n}{2} -1} 2^{2(j)} \big)   $ pebbles to cover the vertices of $C_n$. To cover the vertices $v_{i+2}$ and $v_{i+n-2}$, we need $2^{n-2}$ pebbles. Similarly to cover the vertices $v_{i+4}$ and $v_{i+n-4}$, we need $2^{n-4}$ pebbles. We use $2^{\frac{n}{2}+1}$ pebbles to cover $v_{i+2(\lfloor {\frac{n}{4}} \rfloor )}$ and $v_{i+{2(\lfloor {\frac{n}{4}} \rfloor  +1) }}$. Obviously, we require a pebble to cover  ${\dot{s}}$. %=v_i (i= 1,2,3,\cdots,n)$.
		Thus,   $\mu_{vc} (C_n) \geq   2\big(\sum_{j = {\lfloor {\frac{n}{4}} \rfloor} + 1}^{\frac{n}{2} -1} 2^{2(j)} \big)  + 1  $.\\
	
	\begin{table}[h]
	\centering 
	\begin{tabular}{|c|c|c|c|c|c|c|c|c|c|c|c|} \hline
		&$v_i$	&$v_{i+2}$	&$v_{i+4}$	&$\cdots$ &$v_{{\frac{n}{2}+i-3}}$  &$v_{{\frac{n}{2}+i-1}}$	&$v_{\frac{n}{2}+i+1}$ &$v_{{\frac{n}{2}+i+3}}$ & $\cdots$ &$v_{n+i-4}$ & $v_{n+i-2}$
		\\ \hline
		$v_i$ &$0$ &$n-2$ &$n-4$ &$\cdots$ &$\frac{n}{2}+3$ &$\frac{n}{2}+1$ &$\frac{n}{2}+1$ &$\frac{n}{2}+3$  & $ \cdots $ &$n-4$ &$n-2$   \\ \hline

		\end{tabular}
		\vspace{0.2cm}
		\caption{Detour monophonic distance from $v_{i} \  to \ V(D^2_{vc})$.}
		\label{tab17}
	\end{table} 
	
	Clearly only one pebble suffices to cover the vertex $v_i (i= 1,2,3,\cdots,n)$. From Table \ref{tab17},  all  the
	vertices of $D^2_{vc}$ are covered by the following series of pebble distribution. 
	
	$ \mu_{vc} (G, v_i) + \mu_{vc} (G, v_{i+2}) +  \mu_{vc} (G, v_{i+4})+ \cdots + \mu_{vc} (G, v_{{\frac{n}{2}+i-3}} ) + \mu_{vc} (G, v_{{\frac{n}{2}+i-1}} )+\mu_{vc} (G, v_{{\frac{n}{2}+i+1}})  +\mu_{vc} (G, v_{{\frac{n}{2}+i+3}})  + \cdots + \mu_{vc} (G, v_{n+i-4})+\mu_{vc} (G, v_{n+i-2}) $. \\
	$ = 2^0+ 
	2^{n-2}+ 2^{n-4}+ \cdots + 2^{\frac{n}{2}+1 }+ 2^{\frac{n}{2}+1 }+ \cdots + 2^{n-2}+2^{n-4} \\
	1 + 2( 2^{n-2}+ 2^{n-4}+\cdots +2^{2({\frac{n}{2} -1})}+ 2^{2({\frac{n}{2} -2})}+\cdots +2^{2(\lfloor {\frac{n}{4}} \rfloor  +1) } )$\\
	Thus,   $  2\big(\sum_{j = {\lfloor {\frac{n}{4}} \rfloor}+ 1}^{\frac{n}{2} -1} 2^{2(j)} \big)  + 1 $.\\
	Since all the vertices of $C_n$ are at equivi distance, the result holds for any $v_i \in V(C_n)$, when $n$ is even and $n \equiv 2 (mod ~4)$. Hence, $\mu_{vc} (C_n) =  2\big(\sum_{j = {\lfloor {\frac{n}{4}} \rfloor} + 1}^{\frac{n}{2} -1} 2^{2(j)} \big)  + 1 $.\\\\
	
	\noindent{\bf Case 3:} 
	When $n$ is odd, assume the detour monophonic cover set $D^3_{vc} = \{ v_{i+j} | i= 1,2,3,\cdots,n; j= 0,1,2,\cdots, n-1 \}$, where all  the indices run over modulus $n$  and  $|D_{vc}|= {\lfloor \frac{n}{2} \rfloor}+1$. Without loss of generality we fix  ${\dot{s}} = v_1$. 
	
	Let the initial configuration be ${\sum_{k = {\lfloor \frac {n}{2} \rfloor}+1}^{n-2} 2^{k}} + 2$.  The vertices $v_3$, $v_5$  and $v_7$ are covered with $2^{n-2}$, $2^{n-4}$ and $2^{n-5}$ pebbles respectively. To cover the vertices  $v_{ {\lfloor  {\frac{n}{2} } \rfloor } -2}$, $v_{ {\lfloor  {\frac{n}{2} } \rfloor } }$ and  $v_{ {\lfloor   {\frac{n}{2} } \rfloor } +2}$, we have to use $2^{\lfloor {\frac{n}{2}\rfloor }+4}$, $2^{\lfloor {\frac{n}{2}\rfloor }+2}$ and  $2^{{\lfloor {\frac{n}{2}\rfloor }+1}}$ pebbles on the corresponding vertices. Continuing in this manner we are left with $v_1$ uncovered. Hence, the monophonic vertex cover pebbling number
	
	$\mu_{vc} (C_n) \geq {\sum_{k = {\lfloor \frac {n}{2} \rfloor}+1}^{n-2} 2^{k}} + 3 $. \\
	If $n$ is odd and $n \equiv 1 (mod ~4)$, then $D^3_{vc}=\{	v_1, v_{3} , v_{5}, \cdots, v_{\lfloor{\frac{n}{2}\rfloor}-1} , v_{\lfloor{\frac{n}{2}\rfloor}+1}, v_{\lfloor{\frac{n}{2}\rfloor}+3}, v_{\lfloor{\frac{n}{2}\rfloor}+5}, \cdots ,  \\  v_{n-4} , v_{n-2}, v_{n}  \}$. Without loss of generality consider $v_1$ to be the source vertex. 
	
	\begin{table}[h]
		\centering 
		\begin{tabular}{|c|c|c|c|c|c|c|c|c|c|c|c|c|c|c|} \hline
			&$v_1$	&$v_{3}$	&$v_{5}$ 	&$\cdots$ &$v_{\lfloor{\frac{n}{2}\rfloor}-1}$	&$v_{\lfloor{\frac{n}{2}\rfloor}+1}$ &$v_{\lfloor{\frac{n}{2}\rfloor}+3}$ &$v_{\lfloor{\frac{n}{2}\rfloor}+5}$  &$\cdots$ &$v_{n-4}$ &$v_{n-2}$ & $v_{n}$ 
			\\ \hline
			$v_1$ &$0$ &$n-2$ &$n-4$  &$\cdots$ 
			&$\lfloor{\frac{n}{2}\rfloor}+3$ 
			&${\lfloor{ \frac {n}{2}\rfloor}+1}$ &${\lfloor{\frac{n}{2}\rfloor}+2}$  &${\lfloor{\frac{n}{2}\rfloor}+4}$  &$ \cdots $ &$n-5$ &$n-3$ & $1$  \\ \hline
		\end{tabular}
		\vspace{0.2cm}
		\caption{Detour monophonic distance from $v_{1} \  to \ V(D^3_{vc})$.}
		\label{tab18}
	\end{table} 

	If $n$ is odd and $n \equiv 3 (mod ~4)$, then $D^{3*}_{vc}=\{	v_1, v_{3} , v_{5}, \cdots, v_{\lfloor{\frac{n}{2}\rfloor}-2} , v_{\lfloor{\frac{n}{2}\rfloor}}, v_{\lfloor{\frac{n}{2}\rfloor}+2}, v_{\lfloor{\frac{n}{2}\rfloor}+4}, \cdots ,  \\  v_{n-4} , v_{n-2}, v_{n}  \}$. Without loss of generality let $v_1$  be the source vertex. 
	
	\begin{table}[h]
		\centering 
		\begin{tabular}{|c|c|c|c|c|c|c|c|c|c|c|c|c|c|} \hline
			&$v_1$	&$v_{3}$	&$v_{5}$ 	&$\cdots$ &$v_{\lfloor{\frac{n}{2}\rfloor}-2}$	&$v_{\lfloor{\frac{n}{2}\rfloor}}$ &$v_{\lfloor{\frac{n}{2}\rfloor}+2}$ &$v_{\lfloor{\frac{n}{2}\rfloor}+4}$  &$\cdots$ &$v_{n-4}$ &$v_{n-2}$ & $v_{n}$ 
			\\ \hline
			$v_1$ &$0$ &$n-2$ &$n-4$  &$\cdots$ 
			&$\lfloor{\frac{n}{2}\rfloor}+4$ 
			&${\lfloor{ \frac {n}{2}\rfloor}+2}$ &${\lfloor{\frac{n}{2}\rfloor}+1}$  &${\lfloor{\frac{n}{2}\rfloor}+3}$  &$ \cdots $ &$n-5$ &$n-3$ & $1$  \\ \hline
		\end{tabular}
		\vspace{0.2cm}
		\caption{Detour monophonic distance from $v_{1} \  to \ V(D^{3*}_{vc})$.}
		\label{tab18*}
	\end{table} 
	
	In order to cover the vertices of $C_n$, it suffices to cover the vertices of   vertex cover set.
	From Tables \ref{tab18} and \ref{tab18*},   the pebbles are distributed sums up as follows.\\

	$=2^0+2^{1} +2^ {\lfloor{\frac{n}{2}\rfloor}+1 }
	+2^{\lfloor{ \frac {n}{2}\rfloor}+2} +2^{\lfloor{\frac{n}{2}\rfloor}+3}  +2^{\lfloor{\frac{n}{2}\rfloor}+4} + \cdots + 2^{n-5}+2^{n-4}+2^{n-3}+ 2^{n-2} $\\
$	={\sum_{k = {\lfloor \frac {n}{2} \rfloor}+1}^{n-2} 2^{k}} + 3 $. \\ Thus,
	$\mu_{vc} (C_n) ={\sum_{k = {\lfloor \frac {n}{2} \rfloor}+1}^{n-2} 2^{k}} + 3 $. \\
		Since the vertices of $C_n$ are  equivi distanced, the result is true for any $v_i \in \{ v_i | 2 \leq i \leq n \}$, when $n$ is odd. 
	
	Hence proved.\\  
   \end{proof}

%%%%%%%% Wheel $$$$$$$$$$$$$$$$$$$$$$$$$
   \begin{theorem} \label{thm5}
	For a  Wheel $W_n$, \\ ${\mu_{vc}}(W_n)$ =
	$\begin{cases}
		 {\sum_{i = \frac {n}{2} }^{n-3} 2^{i}} + 5 ,  &  \text{if } n \ is \ \text{even }\\
		 	2\big(\sum_{j = \lfloor{\frac{n}{4}\rfloor} + 1}^{{\lfloor{\frac{n}{2}}\rfloor} -1} 2^{2(j)} \big) + 2^{\lfloor{\frac{n}{2}}\rfloor} + 3, & \text{if } n \ is \ \text{odd }\  \text{and }   n \equiv 1 (mod ~4)\\
			2\big(\sum_{k = {\lfloor {\frac{n}{4}} \rfloor} + 1}^{\lfloor{\frac{n}{2}}\rfloor -1} 2^{2(k)} \big)  + 3, & \text{if } n \ is \ \text{odd} \ \text{and } 
		\ n \equiv 3 (mod ~4)\\

			\end{cases}$
 	       \end{theorem}
	
			\begin{proof}
	
		Let $W_n= C_{n-1}+v_0$ be a wheel of order $n $ with $V(W_n)= \{v_i | \ 0 \leq i \leq n-1 \}$ with $|V(W_n)|=n$. \\
	
	\noindent{\bf Case 1:} 
When $n$ is even,
assume the detour monophonic cover set $D'_{vc} = \{ v_0, v_{i+j} | i= 1,2,3,\cdots,n-1; j= 0,1,2,\cdots, n-2 \}$, where all  the indices run over modulus $n$  and  $|D'_{vc}|= { \frac{n}{2} }+1$. Without loss of generality consider ${\dot{s}} = v_1$. 
% type
Let the initial configuration be ${\sum_{i = \frac {n}{2} }^{n-3} 2^{i}} + 4 $ pebbles assigned  to cover the  vertices of $W_n$. Since  the detour monophonic distance between $v_0 $ and the source vertex is 1, we can cover  2 pebbles.  The vertices $v_3$, $v_5$  and $v_7$ are covered with $2^{n
	-3}$, $2^{n-5}$ and $2^{n-7}$ pebbles respectively. Proceeding in this manner we cover the all  vertices of $D'_{vc}$ except the source vertex with ${\sum_{i = \frac {n}{2} }^{n-3} 2^{i}} + 4 $ pebbles. Thus, continuing in this manner we are left with $v_1$ uncovered. Hence, \\
$\mu_{vc} (W_n)  \geq {\sum_{i = \frac {n}{2} }^{n-3} 2^{i}} + 5 $. \\
Now, consider $n$ is even and $n \equiv 0 (mod ~4)$.\\
Let the vertex cover set $ V(D^1_{vc})= \{ v_0, v_1, v_{3}, v_{5}, \cdots, v_{{\frac{n}{2}-3}}, v_{\frac{n}{2}-1}, v_{{\frac{n}{2}+1}}, v_{{\frac{n}{2}+3}}, \cdots, v_{n-3},\\ v_{n-1} \}$, where $|D^1_{vc}|= {\frac{n}{2}}+1$. Fix ${\dot{s}} = v_1$.

\begin{table}[h]
	\centering 
	\begin{tabular}{|c|c|c|c|c|c|c|c|c|c|c|c|c|} \hline
		&$v_0$	&$v_1$	&$v_{3}$	&$v_{5}$	&$\cdots$ &$v_{{\frac{n}{2}-3}}$	&$v_{\frac{n}{2}-1}$ &$v_{{\frac{n}{2}+1}}$ &$v_{{\frac{n}{2}+3}}$   & $\cdots$  &$v_{n-3}$ &$v_{n-1}$ 
		\\ \hline                                                       
		$v_1$ &$1$ &$0$ &$n-3$ &$n-5$ &$\cdots$ &$\frac{n}{2}+3$ &${\frac{n}{2}}+1$ &$\frac{n}{2}$ &${\frac{n}{2}}+2$ &$ \cdots $  &$n-4$ &$1$   \\ \hline
		
	\end{tabular}
	\vspace{0.2cm}
	\caption{Detour monophonic distance from $v_{1} \  to \ V(D^1_{vc})$.}
	\label{tab18a}
\end{table} 

Let $n$ is even and $n \equiv 2 (mod ~4)$.\\
Then  the vertex cover set $ V(D^{1*}_{vc})= \{ v_0, v_1, v_{3}, v_{5}, \cdots, v_{{\frac{n}{2}-2}}, v_{\frac{n}{2}}, v_{{\frac{n}{2}+2}}, v_{{\frac{n}{2}+4}}, \cdots, v_{n-3},\\  v_{n-1} \}$ and $|D^{1*}_{vc}|= {\frac{n}{2} }+1$. Without loss of generality take ${\dot{s}} = v_1$.  \\
\begin{table}[h]
	\centering 
	\begin{tabular}{|c|c|c|c|c|c|c|c|c|c|c|c|c|} \hline
		&$v_0$	&$v_1$	&$v_{3}$	&$v_{5}$	&$\cdots$ &$v_{{\frac{n}{2}-2}}$	&$v_{\frac{n}{2}}$ &$v_{{\frac{n}{2}+2}}$ &$v_{{\frac{n}{2}+4}}$   & $\cdots$  &$v_{n-3}$ &$v_{n-1}$ 
		\\ \hline                                                       
		$v_1$ &$1$ &$0$ &$n-3$ &$n-5$ &$\cdots$  &${\frac{n}{2}}+2$ &$\frac{n}{2}$ &${\frac{n}{2}}+1$ &${\frac{n}{2}}+3$ &$ \cdots $  &$n-4$ &$1$   \\ \hline
	\end{tabular}
	\vspace{0.2cm}
	\caption{Detour monophonic distance from $v_{1} \  to \ V(D^{1*}_{vc})$.}
	\label{tab18a*}
\end{table} 

From Tables \ref{tab18a} and \ref{tab18a*},  in order to cover the vertices of $W_n$, it suffices to cover the vertices of the corresponding  vertex cover set. Thus,  the pebble distribution with either of the above vertex cover sets is given by,\\
%$ 2^{1}+ 2^0+2^{n-3}+2^{n-4}+2^{n-5}+ \cdots + 2^{{\frac{n}{2}}}  +  2^{{\frac{n}{2}}+1}  +  2^{{\frac{n}{2}}+2} +  2^{{\frac{n}{2}}+3} + +\cdots + 2^{n-5}+2^{n-4}+2^{n-3}$.\\
$ 2^{1}+ 2^0+ 2^{{\frac{n}{2}}}  +  2^{{\frac{n}{2}}+1}  +  2^{{\frac{n}{2}}+2} +  2^{{\frac{n}{2}}+3}  +\cdots + 2^{n-5}+2^{n-4}+2^{n-3}$, which  sums up into
${\sum_{i = \frac {n}{2} }^{n-3} 2^{i}} + 5 $.\\ The result is true for all $v_i (2 \leq i \leq n-1 )$.\\
%okfor wheel
Let the source  vertex be the hub.
Then ${\dot{s}= v_{0} }$. Without loss of generality let $D^4_{vc}=\{v_0, v_1, v_3,  \cdots, v_{n-5}, v_{n-3}, v_{n-1}   \}$. \\
 
The detour monophonic distance from any other vertex of $W_n$ to $v_0$ is 1. Thus in order to cover the vertices of $D^4_{vc}$, we have the  pebble distribution that sums up as follows.
$\mu_{vc} (G, v_0) + \mu_{vc} (G, v_1) + \mu_{vc} (G, v_{3}) + \cdots + \mu_{vc} (G,v_{ n-5}) + \mu_{vc} (G,v_{ n-3}) + \mu_{vc} (G,v_{ n-1}) $\\
$= 2^0+ 2^1+2^1+\cdots +2^1+2^1+2^1$\\
%0kfor wheel
$ 2 {( {\frac{n}{2}} ) }+1= n+1 $ which is less than the considered number of pebbles.\\ 

Thus,
$\mu_{vc} (W_n) ={\sum_{i = \frac {n}{2} }^{n-3} 2^{i}} + 5 $. 
Hence proved.\\ 

		\noindent{\bf Case 2:} 
			When $W_n$ be an odd wheel of order $n$.\\ Let $(n \geq 3)$
				and $n \equiv 1 (mod ~4)$.\\
		Consider the detour monophonic vertex cover set $D''_{vc} = \{ v_0,  v_{2(k)-1} | k = 1,2, 3,\\ \cdots, n-2 \}$ and thus $|D''_{vc}|= {\lfloor \frac{n}{2} \rfloor}$.
		
		Let ${\dot{s}}  = v_1$ and  the initial configuration be $2\big(\sum_{j = \lfloor{\frac{n}{4}\rfloor} + 1}^{{\lfloor{\frac{n}{2}}\rfloor} -1} 2^{2(j)} \big) + 2^{\lfloor{\frac{n}{2}}\rfloor} + 1 $.
		In order to cover the vertices    $v_{(n-3)}$ we require  $2^ {n - 3}$ pebbles  and to cover  $v_{(n-5)}$ we require  $2^ {n -5}$ pebbles respectively using the detour monophonic path. Similar process can be carried on to cover the remaining vertices of $D^1_{vc}$ using the following series of pebble distribution that sums up, \\
		$2^{1}+
		2^{n-3}+ 2^{n-5}+ \cdots + 2^{\lfloor \frac{n}{2} \rfloor+2 }+ 2^{\lfloor \frac{n}{2} \rfloor}+2^{{\lfloor \frac{n}{2} \rfloor}+2 }+ \cdots + 2^{n-5}+2^{n-3} $\\
		$= 2+ 2( 2^{n-3}+ 2^{n-5}+\cdots +2^{{{\lfloor \frac{n}{2} \rfloor}+2 }} )+2^{\lfloor \frac{n}{2} \rfloor}$ \\
		$ = 2\big(\sum_{i = \lfloor{\frac{n}{4}\rfloor} + 1}^{{\lfloor{\frac{n}{2}}\rfloor} -1} 2^{2(i)} \big) + 2^{\lfloor{\frac{n}{2}}\rfloor} + 2 $. \\ 
		Thus we can cover the vertices $v_3$ and $v_5$ with $2^ {n - 3}$ and $2^ {n - 5}$ pebbles and we are left with no pebble to cover the source vertex $v_1$. Therefore, the detour monophonic vertex cover number
				 	$\mu_{vc} (W_n)  \geq  2\big(\sum_{j = \lfloor{\frac{n}{4}\rfloor} + 1}^{{\lfloor{\frac{n}{2}}\rfloor} -1} 2^{2(j)} \big) + 2^{\lfloor{\frac{n}{2}}\rfloor} + 3 $. \\ 
		 	
		 Now let ${\dot{s}}  = v_i (1 \leq i \leq n-1)$ and 	$D^{2}_{vc} = \{ v_0,  v_{i+2}, v_{i+4}, \cdots, v_{\lfloor{\frac{n}{2}\rfloor}+i-2}, v_{\lfloor{\frac{n}{2}\rfloor}+i}, v_{{\lfloor {\frac{n}{2}\rfloor}+i+2}}, \\ \cdots, v_{n+i-5}, v_{n+i-3} \}$, where the indices run over modulus $n$. \\
		 
			\begin{table}[h]
			\centering 
			\begin{tabular}{|c|c|c|c|c|c|c|c|c|c|c|c|} \hline
				&$v_0$ &$v_i$	&$v_{i+2}$	&$v_{i+4}$	&$\cdots$ &$v_{\lfloor{\frac{n}{2}\rfloor}+i-2}$	&$v_{\lfloor{\frac{n}{2}\rfloor}+i}$ &$v_{{\lfloor{\frac{n}{2}\rfloor}+i+2}}$ & $\cdots$ &$v_{n+i-5}$ &$v_{n+i-3}$
				\\ \hline
				$v_i$ &$1$ &$0$ &$n-3$ &$n-5$ &$\cdots$ &$\lfloor{\frac{n}{2}\rfloor}+2$ &$\lfloor{\frac{n}{2}}\rfloor$ &$\lfloor{\frac{n}{2}\rfloor}+2$  &$ \cdots $ &$n-5$ &$n-3$   \\ \hline
							\end{tabular}
			\vspace{0.2cm}
			\caption{Detour monophonic distance from $v_{i}$ \  to \ $V(D^{2}_{vc})$.}
			\label{tab16a}
		\end{table} 
		Obviously $v_i$ is covered with one pebble. From Table 
		\ref{tab16a}, $D_m (v_0,v_i )=1$ and so $v_0$ is covered with two pebbles. In order to cover the vertices   $v_{i+2}$ and $v_{n+i-3}$ we require  $2^ {n - 3}$ pebbles each and to cover  $v_{i+4}$ and $v_{n+i-5}$ we require  $2^ {n - 5}$ pebbles each. 
		Similarly  to cover $v_{\lfloor{\frac{n}{2}\rfloor}+i-2}$ and $v_{\lfloor{\frac{n}{2}\rfloor}+i-2}$, we require $ 2^ {\frac{n}{2} +2}$ pebbles. We need $2^ {\lfloor{\frac{n}{2}}\rfloor}$  
		pebbles to cover the vertex $v_{{\lfloor{\frac{n}{2}\rfloor}+i}}$. Thus in order to cover the vertices in $D^2_{vc}$ we distribute the pebbles as follows.
		$ \mu_{vc} (G, v_0) +\mu_{vc} (G, v_i) + \mu_{vc} (G, v_{i+2}) +  \mu_{vc} (G, v_{i+4})+ \cdots + 
		        \mu_{vc}(G, v_{{\lfloor{\frac{n}{2}}\rfloor}+i-2} )+
				\mu_{vc} (G, v_{{\lfloor{\frac{n}{2}} \rfloor}+i} ) + 
				\mu_{vc} (G, v_{i+{\frac{n+4}{2} }+2} ) +
			   \cdots +\mu_{vc} (G, v_{n+i-5})+\mu_{vc} (G, v_{n+i-3}) $. \\
		$ =2^1+ 2^0+ 
		2^{n-3}+ 2^{n-5}+ \cdots + 2^{\lfloor{\frac{n}{2}}\rfloor+2}+ 2^{\lfloor{\frac{n}{2}}\rfloor}+2^{\lfloor{\frac{n}{2}}\rfloor+2}+ \cdots + 2^{n-5}+2^{n-3} $\\
		$=2+ 1 + 2( 2^{n-3}+ 2^{n-5}+\cdots +2^{\lfloor{\frac{n}{2}}\rfloor +2} )+2^{\lfloor{\frac{n}{2}}\rfloor}$ \\
		$=2\big(\sum_{j = \lfloor{\frac{n}{4}\rfloor} + 1}^{{\lfloor{\frac{n}{2}}\rfloor} -1} 2^{2(j)} \big) + 2^{\lfloor{\frac{n}{2}}\rfloor} + 3 $. \\ 
	
				Since all the vertices are of same detour monophonic distance, the number of pebbles required to cover from any other vertex of $W_n$ except the hub is  	$ 2\big(\sum_{j= \lfloor{\frac{n}{4}\rfloor} + 1}^{{\lfloor{\frac{n}{2}}\rfloor} -1} 2^{2(j)} \big) + 2^{\lfloor{\frac{n}{2}}\rfloor} + 3 $. \\ 			
				
	Assume the source  vertex is the hub.
		
		Then ${\dot{s}= v_{0} }$. Without loss of generality let $D*_{vc}=\{v_0, v_1, v_3,  \cdots, v_{n-6}, v_{n-4}, v_{n-2}  \}$. \\
	
			Since $D_m (x, v_0)= 1$, where $x$ is any  vertex of $W_n$,    Thus in order to cover the vertices of $D*_{vc}$, we have the  pebble distribution that sums up as follows.
		$\mu_{vc} (G, v_0) + \mu_{vc} (G, v_1) + \mu_{vc} (G, v_{3}) + \cdots + \mu_{vc} (G,v_{ n-6}) + \mu_{vc} (G,v_{ n-4}) + \mu_{vc} (G,v_{ n-2}) $\\
		$= 2^0+ 2^1+2^1+\cdots +2^1+2^1+2^1$\\
					$=2 { \lfloor{\frac{n}{2}} \rfloor }+1$ which is less than the considered number of pebbles.\\ 
			Thus,
		$\mu_{vc} (W_n)  =  2\big(\sum_{j = \lfloor{\frac{n}{4}\rfloor} + 1}^{{\lfloor{\frac{n}{2}}\rfloor} -1} 2^{2(j)} \big) + 2^{\lfloor{\frac{n}{2}}\rfloor} + 3 $.\\
		Hence proved.\\ 
		
		\noindent{\bf Case 3:} 
	
		When $n$ is odd and $n \equiv 3 (mod ~4)$, let $D^3_{vc} = \{ v_{i+ 2j} |i= 1,2,3,\cdots,n; j= 0,1,2,\cdots, \lfloor { \frac{n}{4}-1}	\rfloor,  \lfloor {\frac{n}{4}} \rfloor,  \lfloor { \frac{n}{4}+1} \rfloor, \cdots,  { \frac{n}{2}-2}, { \frac{n}{2}-1} \}$ be the vertex cover set, where the indices run over modulus $n$. Thus, $|D^3_{vc}|= \frac{n}{2}$. 
		Now consider the configuration  which encages $2\big(\sum_{k} 2^{2(k)} \big)  + 2 $ pebbles to cover the vertices of $W_n$. To cover the vertices $v_{i+2}$ and $v_{i+n-2}$, we need $2^{n-2}$ pebbles. Similarly to cover the vertices $v_{i+4}$ and $v_{i+n-4}$, we need $2^{n-4}$ pebbles. We use $2^{\frac{n}{2}+1}$ pebbles to cover $v_{i+2(\lfloor {\frac{n}{4}} \rfloor )}$ and $v_{i+{2(\lfloor {\frac{n}{4}} \rfloor  +1) }}$. Obviously, we require a pebble to cover  ${\dot{s}}=v_i (i= 1,2,3,\cdots, n-1)$.
		
		Thus,    $\mu_{vc} (W_n) \geq  2\big(\sum_{k = {\lfloor {\frac{n}{4}} \rfloor} + 1}^{\lfloor{\frac{n}{2}}\rfloor -1} 2^{2(k)} \big)  + 3 $.\\
		
		\begin{table}[h]
		\centering 
		\begin{tabular}{|c|c|c|c|c|c|c|c|c|c|c|c|} \hline
			&$v_i$	&$v_{i+2}$	&$v_{i+4}$	&$\cdots$ &$v_{\lfloor{\frac{n}{2}\rfloor}+i-3}$
			&$v_{\lfloor{\frac{n}{2}\rfloor}+i-1}$	&$v_{\lfloor{\frac{n}{2}\rfloor}+i+1}$ &$v_{{\lfloor{\frac{n}{2}\rfloor}+i+3}}$ & $\cdots$ &$v_{n+i-5}$ &$v_{n+i-3}$
			\\ \hline
			$v_i$ &$0$ &$n-3$ &$n-5$ &$\cdots$ &$\lfloor{\frac{n}{2}\rfloor}+3$
			&$\lfloor{\frac{n}{2}\rfloor}+1$ &$\lfloor{\frac{n}{2}}\rfloor +1$ &$\lfloor{\frac{n}{2}\rfloor}+3$  &$ \cdots $ &$n-5$ &$n-3$   \\ \hline
			
			\end{tabular}
			\vspace{0.2cm}
			\caption{Detour monophonic distance from $v_{i} \  to \ V(D^3_{vc})$.}
			\label{tab17a}
		\end{table} 
		
		Clearly only one pebble suffices to cover the vertex $v_i (i= 1,2,3,\cdots,n)$. From Table \ref{tab17a},  all  the
		vertices of $D^2_{vc}$ are covered by the following series of pebble distribution. 
		
		$ \mu_{vc} (G, v_0) + \mu_{vc} (G, v_i) + \mu_{vc} (G, v_{i+2}) +  \mu_{vc} (G, v_{i+4})+ \cdots + \mu_{vc} (G, v_{\lfloor{\frac{n}{2}\rfloor}+i-3} )+ \\ \mu_{vc} (G, v_{\lfloor{\frac{n}{2}\rfloor}+i-1}) + \mu_{vc} (G, v_{\lfloor{\frac{n}{2}\rfloor}+i+1} )+\mu_{vc} (G, v_{\lfloor{\frac{n}{2}\rfloor}+i+3})  + \cdots + \mu_{vc} (G, v_{n+i-5})+\mu_{vc} (G,v_{n+i-3}) $. \\
		$ = 2^1+ 2^0+ 
		2^{n-3}+ 2^{n-5}+ \cdots + 2^{\lfloor{\frac{n}{2}}\rfloor +3 }+ 2^{\lfloor{\frac{n}{2}}\rfloor +1 } + 2^{\lfloor{\frac{n}{2}}\rfloor +1 }+ 2^{\lfloor{\frac{n}{2}}\rfloor +3 }+ \cdots + 2^{n-5}+2^{n-3}$ \\
		=$2\big(\sum_{k = {\lfloor {\frac{n}{4}} \rfloor} + 1}^{\lfloor{\frac{n}{2}}\rfloor -1} 2^{2(k)} \big)  + 3 $. \\
				Since all the vertices except $v_0$ of $W_n$ are at equivi distance, the result holds for any $v_i \in V(W_n) (2 \leq i \leq n-1)$, when $n$ is odd and $n \equiv 3 (mod ~4)$.
		
				When the source  vertex is the hub, the result follows from the previous discussion of this Theorem and the configuration of pebbles is $2 { \lfloor{\frac{n}{2}} \rfloor }+1$ which is less than the considered number of pebbles.\\ Hence, \\
					 $\mu_{vc} (W_n) =  2\big(\sum_{k = {\lfloor {\frac{n}{4}} \rfloor} + 1}^{\lfloor{\frac{n}{2}}\rfloor -1} 2^{2(k)} \big)  + 3 $.\\
				
			\end{proof}
				
		\section{\bf Algorithm to find DMVCPN}
		
		\noindent{\bf Step 1:} Find the vertex cover set $D_{vc}$ with minimum cardinality for a graph $G$, which covers all the edges of $G$. 
		
		\noindent{\bf Step 2:} Fix the source vertex $\dot{s}$.
			
		\noindent{\bf Step 3:} Find the detour monophonic distance  $D_m(\dot{s}, v)$, where $v \in V(G)$.
		
		\noindent{\bf Step 4:}  By distributing the pebbles place one pebble each to cover the vertices of $D_{vc}$ such that the number of pebbles used is minimum for all configuration of pebbles. 
	
		\noindent{\bf Step 5:} Repeat Step 4 for various source vertices until reaching the minimal number of pebbles that cover all the vertices of $D_{vc}$ and so the vertices of $G$.
	
		\noindent{\bf Step 6:} End the process if each  vertex of $G$ is considered to be $\dot{s}$. \\

			\section{\bf Various Pebbling Numbers of $G$}
			
			The various graph pebbling numbers displayed in  the Table \ref{tab24} exhibit that DMVCPN of a graph is distinct.\\
			
			\begin{table}[h]
			\centering 
			\tiny
			\begin{tabular}{cccc} \hline
			Standard 	&  Detour pebbling 	& Monophonic pebbling  	& Detour monophonic vertex  cover 
				\\ 
		     Graph  $G$       & number  $f^{*}(G) $ \cite{5b}    	&  number $\mu(G)$ \cite{5}   	&  pebbling number $\mu_{vc}(G)$
			\\ \hline \hline
			
				  $P_n$   &$2^{n-1}$       &$2^{n-1} $      &	$\sum_{i = 0}^{ \lfloor {\frac{n}{2}} \rfloor -1} 2^{2(i)+1},$	n  is  odd  \\ 
				  $$   &$$       &$ $      & $ \sum_{j = 0}^{ {\frac{n}{2}}  -1} 2^{2(j)}$, n is even \\ 
				  
				$C_n$     &$2^{n-1} $       & $2^{n-2} +1$     &$	\begin{cases}
					2\big(\sum_{i = \frac{n}{4} + 1}^{\frac{n}{2} -1} 2^{2(i)} \big) + 2^{\frac{n}{2} } + 1, & \text{if } n \ is \ \text{even } \text{and }   n \equiv 0 (mod ~4)\\
					
					2\big(\sum_{j = { \lfloor \frac{n}{4} \rfloor } + 1}^{\frac{n}{2}-1} 2^{2(j)} \big) +  1, & \text{if } n \ is \ \text{even } \text{and } 
					\ n \equiv 2 (mod ~4)\\
					{\sum_{k = {\lfloor \frac {n}{2} \rfloor}+1}^{n-2} 2^{k}} + 3,  &  \text{if } n \ is \ \text{odd }
					
				\end{cases}$ \\

			$F_n$     &$2^{n-1} $       & $2^{n-2} +1$     &$\sum_{j = 0}^{\lfloor {\frac{n}{2}} \rfloor -1}    2^{2(j)} +  2
			$, n  is odd \\ 
				 $ $   &$ $       &$ $      &	$ \sum_{i = 0}^{  {\frac{n}{2}} -2} 2^{2(i)+1} +2$, n  is  even   \\ 
				 
				$W_n$     &$2^n$      &$2^{n-2} +2$                             & $\begin{cases}
					{\sum_{i = \frac {n}{2} }^{n-3} 2^{i}} + 5 ,  &  \text{if } n \ is \ \text{even }\\
					2\big(\sum_{j = \lfloor{\frac{n}{4}\rfloor} + 1}^{{\lfloor{\frac{n}{2}}\rfloor} -1} 2^{2(j)} \big) + 2^{\lfloor{\frac{n}{2}}\rfloor} + 3, & \text{if } n \ is \ \text{odd }\  \text{and }   n \equiv 1 (mod ~4)\\
					2\big(\sum_{k = {\lfloor {\frac{n}{4}} \rfloor} + 1}^{\lfloor{\frac{n}{2}}\rfloor -1} 2^{2(k)} \big)  + 3, & \text{if } n \ is \ \text{odd} \ \text{and } 
					\ n \equiv 3 (mod ~4)\\
					
				\end{cases}$\\

				   \hline \hline
										
			\end{tabular}
			\vspace{0.2cm}
			\caption{ Various pebbling numbers for some standard graphs}
			\label{tab24}
		\end{table} 
		
	\section{\bf Conclusion}
		 	In this article, we have combined   cover pebbling number and detour monophonic number to obtain an interesting variate of graph theory, the detour monophonic vertex cover pebbling number
		 		the interesting areas of graph theory. We have  studied    and determined $\mu_{vc} (G)$  for some standard graphs. To evaluate the  specific values of $\mu_{vc} (G)$ for other classes of well-known graphs is an open area of research.

\end{document}